\documentclass[english,twoside,11pt]{article}

\usepackage[german,english]{babel}
\usepackage{amsmath,amsthm,amssymb,latexsym}
\usepackage{moreverb}
\usepackage{url}
\usepackage{graphicx}
\usepackage{longtable}
\usepackage{lscape}
\usepackage{booktabs}
\usepackage{tabularx}
\usepackage{psfrag}
\usepackage{multicol}

\usepackage{paralist}
\usepackage[colorlinks=true]{hyperref}
\usepackage{mathptmx}

\newtheorem{thm}{\bf Theorem}[section]

\newtheorem{thmnonumber}{\bf Main Theorem}

\newtheorem{prop}[thm]{\bf Proposition}
\newtheorem{lemma}[thm]{\bf Lemma}
\newtheorem{corollary}[thm]{\bf Corollary}

\newtheorem{remark}[thm]{\bf Remark}

\def \link{\operatorname{link}}
\def \sd{\operatorname{sd}}
\def \del{\operatorname{del}}
\makeatletter

\title{Knots in collapsible and non-collapsible balls}

\author{\Large Bruno Benedetti\footnote{Supported by the Swedish Research Council, grant ``Triangulerade M{\aa}ngfalder, 
Knutteori i diskrete Morseteori'' and the DFG Collaborative Research Center TRR 109,
``Discretization in Geometry and DynamicsÕ'.}\,\, and Frank H.~Lutz\footnote{Supported by the DFG Research Group ``Polyhedral Surfaces''.}}

\date{}

\begin{document}

\selectlanguage{english}

\maketitle

\enlargethispage{5.5mm}
\begin{abstract}  We construct the first explicit example of a simplicial $3$-ball $B_{15,66}$ that is not collapsible. It has only $15$ vertices. We exhibit  a second $3$-ball $B_{12,38}$ with $12$ vertices that is collapsible and evasive, but not shellable.  Finally, we present the first explicit triangulation of a $3$-sphere $S_{18, 125}$ (with only 18 vertices) that is not locally constructible. All these examples are based on knotted subcomplexes with only three edges; the knots are the trefoil, the double trefoil, and the triple trefoil, respectively. The more complicated the knot is, the more distant the triangulation is from being polytopal, collapsible, etc. 
Further consequences of our work~are:
\begin{compactenum}[ \rm (1) ]
\item Unshellable $3$-spheres may have vertex-decomposable barycentric subdivisions. \\(This shows the strictness of an implication proven by Billera and Provan.)
\item For $d$-balls, vertex-decomposable implies non-evasive implies collapsible, and for $d=3$ all implications are strict. (This answers a question by Barmak.)
\item Locally constructible $3$-balls may contain a double trefoil knot as a $3$-edge subcomplex. \\(This improves a result of Benedetti and Ziegler.) 
\item Rudin's ball is non-evasive.
\end{compactenum}
\end{abstract}

\vspace{5mm}

\section{Introduction}
\textsc{Collapsibility} is a combinatorial property introduced by Whitehead, and somewhat stronger than contractibility. 
In 1964, Bing proved, using knot theory,  that some triangulations of the $3$-ball are not collapsible~\cite{BING, GOO}.
Bing's  method works as follows. One starts with a finely-triangulated $3$-ball embedded in the Euclidean $3$-space. 
Then one drills a knot-shaped tubular hole inside it, stopping one step before destroying the property of being a $3$-ball. The resulting $3$-ball contains a knot that consists of a single interior edge plus many boundary edges. This interior edge is usually called \emph{knotted spanning}. If the knot is sufficiently complicated (like a \emph{double}, or a \emph{triple} trefoil), Bing's ball cannot be collapsible \cite{BING, GOO}; see also \cite{Benedetti-DMT4MWB}.
In contrast, if the knot is simple enough (like a \emph{single} trefoil), then the Bing ball may be collapsible~\cite{LICKMAR}.

  \begin{figure}
    \begin{center}
      \includegraphics[width=6cm]{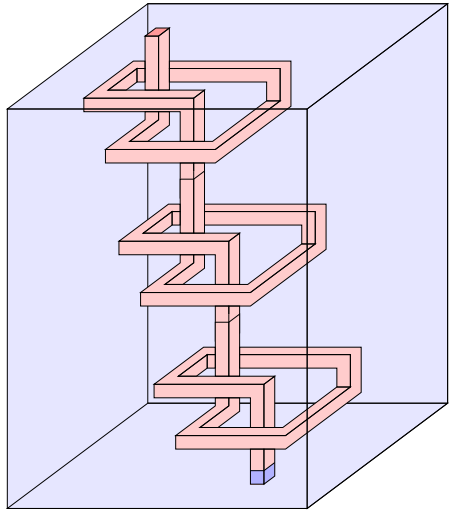}
    \end{center}
    \caption{A triple trefoil drilled inside a ball, stopping one edge before perforating it yields a non-collapsible $3$-ball.
      \label{fig:nonLC}}
  \end{figure}

Thus the existence of a short knot in the triangulation prevents a 3-ball from having a desirable combinatorial property, namely, collapsibility. This turned out to be a recurrent motive in literature. In the Eighties, several authors asked whether all $3$-spheres are shellable. This was answered in 1991 by Lickorish in the negative \cite{LICK}: The presence in a $3$-sphere of a triple trefoil on three edges prevents it from being shellable. It remained open whether all spheres are constructible (a slighly weaker property than shellability). However, in 2000 Hachimori and Ziegler showed that the presence of any non-trivial knot on three vertices in a $3$-sphere even prevents it from being constructible. Finally, in 1994 the physicists Durhuus and Jonsson asked whether all $3$-spheres are locally constructible. Once again, a negative answer, based on Lickorish's original argument, was found using knot theory; see Benedetti--Ziegler \cite{BZ}.

These examples represent spheres that are far away from being polytopal. 
Thus, they are good candidates for testing properties that are true for polytopes, but only conjectured to be true for spheres. Moreover, they represent good test instances for algorithms in computational topology, as they are complicated triangulations of relatively simple spaces. 

Unfortunately, the knotted counterexamples by Lickorish and others have a defect: They are easy to explain at the blackboard, but they yield triangulations with many vertices. 
The purpose of this paper is to come up with analogous `test examples' that are smaller in size, but still contain topological obstructions that prevent them from having nice combinatorial properties.  

A first idea to save on the number of faces is to start by realizing the respective knot in $3$-space, using (curved) arcs. Obviously, any knot can be realized with exactly three arcs in $\mathbb{R}^3$ (we just need to draw it and insert three vertices along the knot). If we thicken the arcs into three `bananas', the resulting $3$-complex~$P$ is homeomorphic to a solid torus pinched three times. By inserting $2$-dimensional membranes, $P$ can be made contractible, and then it can be thickened to a $3$-ball (or a $3$-sphere) simply by adding cones. This approach costs a lot of manual effort, but a posteriori, it allows us to obtain new insight. In fact, here comes the second idea: We can ask a computer to perform random bistellar flips to the triangulation of the ball, \emph{without modifying the subcomplex $P$}. Performing the flips according to a simulated annealing strategy \cite{BjoernerLutz2000} we were able to decrease the size of the triangulation, 
but for sure the flips will preserve the knotted substructure and its number of arcs. 

This construction was introduced by the second author in \cite{LUTZ1}, who applied it to the single trefoil, thereby obtaining a knotted $3$-ball $B_{12, 38}$ with $12$ vertices and $38$ tetrahedra. Here we apply the method to the double trefoil and the triple trefoil. The resulting spheres turn out to be interesting in connection with some properties which we will now describe.

The notion of \textsc{evasiveness} has appeared first in theoretical computer science, in Karp's conjecture on monotone graph properties. Kahn, Saks and Sturtevant~\cite{KahnSaksSturtevant} extended the evasiveness property to simplicial complexes, showing that \emph{non-evasiveness} strictly implies collapsibility. 
One can easily construct explicit examples of collapsible evasive $2$-complexes in which none of the vertex-links is contractible~\cite{BarmakMinian}; see also \cite{BenedettiLutz}. 
Basically there are three known ways to prove that a certain complex $E$ is evasive:
\begin{compactenum}[ $\qquad$ (A) ]
\item One shows that none of its vertex-links is contractible, cf.~\cite{BarmakMinian};
\item one proves that the Alexander dual of $E$ is evasive, cf.~\cite{KahnSaksSturtevant};
\item one shows (for example, via knot-theoretic arguments \cite{BING}) that $E$ is not even collapsible.
\end{compactenum}
\noindent But are there collapsible evasive {\em balls}? And if so, how do we prove that they are evasive? Clearly, none of the approaches above would work. This was asked to us by Barmak (private communication). Once again, we found a counterexample in the realm of knotted triangulations: specifically, Lutz's triangulation  $B_{12, 38}$, which contains a single-trefoil knotted spanning edge.

\begin{thmnonumber} \label{mainthm:1}
The $3$-ball $B_{12, 38}$ is collapsible and evasive. However, it is not shellable and not locally constructible.
\end{thmnonumber}

To prove collapsibility, we tried, using the computer, several collapsing sequences, until we found a lucky one. To show evasiveness, we used some sort of `trick': We computed the homology of what would be left from $B_{12, 38}$ after deleting roughly half of its vertices. It turns out that deleting five vertices from $B_{12, 38}$  (no matter which ones) yields almost always some complex with non-trivial homology. From that we were able to exclude non-evasiveness.

En passant, we also prove the non-evasiveness of other existing triangulations that were known to be collapsible, like Rudin's ball (Theorem \ref{thm:Rudin}) or Lutz's triangulations $B_{7,10}$ \cite{LutzEGnonVD} 
and $B_{9,18}$ \cite{LutzEGnonSH}.

Main Theorem \ref{mainthm:1} can be viewed as an improvement on the  result from 1972 by Lickorish--Martin~\cite{LICKMAR} and Hamstrom--Jerrard \cite{HAM} that a ball with a knotted spanning edge can be collapsible. Recently Benedetti--Ziegler~\cite{BZ} constructed a similar example with all vertices on the boundary. In contrast, our $B_{12, 38}$ has exactly one interior vertex. 
We also mention that $B_{12,38}$ is the first example of  a manifold that admits a perfect discrete Morse function, but cannot admit a perfect Fourier--Morse function in the sense of Engstr\"{o}m \cite{Engstroem}. In fact, a complex is non-evasive if and only if it admits a Fourier--Morse function with only one critical cell. 

\medskip
\textsc{Vertex-decomposability} is a strengthening of shellability, much like non-evasiveness is a strengthening of collapsibility. 
It was introduced by Billera and Provan in 1980, in connection with the Hirsch conjecture \cite{PB}. For $3$-balls, we have the following diagram of implications: 
\[
\begin{array}{ccc}
\hbox{vertex-decomposable} & \Rightarrow & \hbox{shellable} \\

\Downarrow & \qquad & \Downarrow \\

\hbox{non-evasive} & \Rightarrow & \hbox{collapsible} \\

\end{array}
\]

In addition, the barycentric subdivision of any shellable complex is vertex-decomposable~\cite{PB} --- and the barycentric subdivision of any collapsible complex is non-evasive~\cite{Welker}. What about the converse? Can an unshellable ball  or sphere become 
vertex-decomposable after a single barycentric subdivision? The answer is positive. The barycentric subdivision of $B_{12, 38}$ is, in fact, vertex-decomposable. The same holds for $S_{13, 56}$, the unshellable $3$-sphere obtained coning off the boundary of $B_{12, 38}$;
see Proposition~\ref{prop:subdivisions}.

Next, we turn to a concrete question from \textsc{discrete quantum gravity}. Suppose that we wish to take a walk on the various triangulations of $S^3$, by starting with the boundary of the $4$-simplex and performing a random sequence of bistellar flips (also known as `Pachner moves'). All triangulated $3$-spheres can be obtained this way \cite{Pachner}, but some may be less likely to appear than others, like the $16$-vertex triangulation $S_{16,104}$ by Dougherty, Faber and Murphy \cite{BagchiDatta11, DoughertyFaberMurphy}. 
(In fact, any `Pachner walk' from the boundary of the $4$-simplex to $S_{16,104}$ must pass through spheres with more than $16$ vertices.) This `random Pachner walk' model is used in discrete quantum gravity, by Ambj\o rn, Durhuus, Jonsson and others, to estimate the total number of triangulations of $S^3$ \cite{ADJ, AmbVar}. Durhuus and Jonsson have also developed the property of local constructibility, conjecturing it would hold for all $3$-spheres \cite{DJ}. As we said, the conjecture was negatively answered in \cite{BZ}, but it remained unclear how difficult it is to reach counterexamples, using a random Pachner walk. In other words: 
How outspread should the simulation be, before we have the chance to meet a non-locally constructible sphere? 

Here we answer this question by presenting the first explicit triangulation of a non-locally constructible $3$-sphere. For that, we  have to adapt the construction of $B_{12,38}$ from the single trefoil to the triple trefoil. In the end, we manage to use only $18$ vertices. The surprise is that via Pachner moves, the final triangulation is reachable rather straightforwardly.

\begin{thmnonumber}
Some $17$-vertex triangulation $B_{17,95}$ of the $3$-ball contains a triple trefoil knotted spanning edge. This $B_{17,95}$ is not collapsible.
Coning off the boundary of $B_{17,95}$ one obtains a knotted \-{$3$-sphere} $S_{18,125}$ that is \emph{not} locally constructible. Removing \emph{any} tetrahedron from $S_{18,125}$ one obtains a knotted {$3$-ball} that is neither locally constructible nor collapsible. This $S_{18, 125}$ is `$3$-stellated', in the notation of Bagchi--Datta \cite{BagchiDatta11}: it can be reduced to the boundary of a $4$-simplex by using $94$ Pachner moves that do not add further vertices.
\end{thmnonumber}

\medskip

After dealing with the single trefoil and the triple trefoil, let us turn to the intermediate case of the double trefoil. 
By the work of Benedetti--Ziegler, any $3$-ball containing a $3$-edge knot in its $1$-skeleton cannot be locally constructible if the knot is the sum of three or more trefoils~\cite{BZ}. But is this bound best possible? In \cite{BZ} it is shown with topological arguments that a \emph{collapsible} $3$-ball may contain a double trefoil knot on $3$ edges. Recall that locally constructible $3$-balls are characterized by the property of collapsing onto their boundary minus a triangle \cite{BZ}. This is stronger than just being collapsible.
It remained unclear whether a \emph{locally constructible}  $3$-ball may indeed contain a double trefoil on three edges.

\enlargethispage*{2mm}

We answer this question affirmatively in Section~\ref{sec:2}. As before, the key consists in triangulating cleverly, so that computational approaches may succeed. On the way to this result, we produce a smaller example of a non-collapsible ball, using only $15$ vertices and $66$ tetrahedra.

\begin{thmnonumber}
Some $15$-vertex triangulation $B_{15,66}$ of the $3$-ball 
contains a double trefoil knotted spanning edge. This $B_{15,66}$ 
is not collapsible. Coning off the boundary of $B_{15,66}$ 
one obtains a knotted \mbox{$3$-sphere} $S_{16,92}$ 
that \emph{is} locally constructible. Removing the tetrahedron 
$1\,9\,14\,15$ from $S_{16,92}$ one obtains a knotted {$3$-ball} 
that is collapsible and locally constructible.
\end{thmnonumber}

Now, for each $d \ge 3$ one has the following hierarchy 
of combinatorial properties of triangulated $d$-spheres~\cite{BZ}: 
\[ 
 \{ \textrm{vertex-decomposable} \} \subsetneq 
 \{ \textrm{shellable} \}  \subseteq
 \{ \textrm{constructible}\} \subsetneq 
 \{ \textrm{LC}\} \subsetneq
 \{ \textrm{all $d$-spheres}\}.
\] 
\noindent An analogous hierarchy holds for $d$-balls ($d\ge 3$)~\cite{BZ}: 
\[ 
 \{ \textrm{vertex-decomp.} \} \subsetneq 
 \{ \textrm{shellable} \}  \subsetneq
 \{ \textrm{constructible}\} \subsetneq 
 \{ \textrm{LC}\} \subsetneq
\left\{ \!\!\!
\begin{array}{c}
\textrm{collapsible onto} \\
\textrm{$(d-2)$-complex}
\end{array}
\!\!\!\right\} 
\subsetneq 
\{ \textrm{all $d$-balls}\}. 
\] 
\noindent (When $d=3$, ``collapsible onto a $1$-complex'' is equivalent 
to ``collapsible''.)

\pagebreak

Here is another interesting hierarchy for balls, which can 
be merged with the previous one.

\begin{thmnonumber}
 There are the following inclusion relations between families of 
 simplicial $d$-balls: %{\em
\[ 
 \{ \mbox{vertex-decomposable} \} \subseteq
 \{ \mbox{non-evasive} \}  \subseteq
 \{ \mbox{collapsible}\} \subseteq 
 \{ \mbox{all $d$-balls}\}. 
\] 
%}
For $2$-balls all inclusions above are equalities, whereas for $3$-balls 
all inclusions above are strict. More precisely, we have the following 
`mixed' hierarchy:
%{\em
\[ 
 \{ \textrm{vertex-decomposable} \} \subsetneq 
\left\{ \!\!\!
\begin{array}{c}
\mbox{shellable AND} \\
\mbox{non-evasive}
\end{array}
\!\!\!\right\} 
\subsetneq
\left\{ \!\!\!
\begin{array}{c}
\mbox{shellable OR} \\
\mbox{non-evasive}
\end{array}
\!\!\!\right\} 
\subsetneq 
\{ \mbox{collapsible}\} \subsetneq 
 \{ \mbox{all $3$-balls}\}. 
\] 
%}
\end{thmnonumber}

\begin{table}[t]
\caption{List of $3$-balls and $3$-spheres discussed here} 
\vskip3.5mm
\centering 
\begin{tabular}{c l l l l l l}
\hline
Trefoils \vline &  \multicolumn{2}{c}{$3$-ball $B$} \vline & \multicolumn{2}{c}{$3$-Sphere $\partial (v \ast B)$} \vline & 
\multicolumn{2}{c}{$3$-ball $\partial (v \ast B) - \Sigma$} \\
\hline
\\[-3.5mm]
0 & $B_{7,10}$ & \footnotesize{sh., NE, non-VD}& $S_{8,20}$ & \footnotesize{VD}& $B_{8,19}$ & \footnotesize{VD}\\ 
0 & $B_{8,13}$ & \footnotesize{sh., non-VD} & $S_{9,25}$ & \footnotesize{sh., non-VD} & $B_{9,24}$ &\footnotesize{sh.}\\  
0 & $B_{9,18}$ & \footnotesize{constr., NE, non-sh.} & $S_{10,32}$ & \footnotesize{sh.} & $B_{10,31}$ &\footnotesize{sh.}\\ 
1 & $B_{12,38}$ & \footnotesize{coll., evasive, non-LC} & $S_{13,56}$ & \footnotesize{LC, non-constr.} & $B_{13,55}$ &\footnotesize{LC, non-constr.}\\ 
2 & $B_{15,66}$ 
& \footnotesize{non-coll.} & $S_{16,92}$ & \footnotesize{LC, non-constr.} & $B_{16,91}$ &\footnotesize{LC, non-constr.}\\ 
3 & $B_{17,95}$ 
& \footnotesize{non-coll.} & $S_{18,125}$ & \footnotesize{non-LC} & $B_{18,124}$ &\footnotesize{non-coll.}\\ [1ex]
\hline 
\end{tabular}
\vskip1mm
{\footnotesize
 Note: VD = vertex-decomposable, sh. = shellable, constr. = constructible, LC = locally constructible, coll. = collapsible, NE~=~non-evasive. ``\textsc{Trefoils:} $t$'' means ``containing a $t$-fold trefoil on $3$ edges''.}
\label{table:nonlin} 
\end{table}

\section{Background}

\subsection{Combinatorial properties of triangulated spheres and balls}

A $d$-complex is \emph{pure} if all of its top-dimensional faces (called \emph{facets}) have the same dimension.

A pure $d$-complex $C$ is \emph{constructible} if either $C$ is a simplex, or $C$ is a disjoint union of points, or 
$d \ge 1$ and $C$ can be written as $C= C_1 \cup C_2$, where $C_1$ and $C_2$ are constructible $d$-complexes and $C_1 \cap C_2$ is a constructible $(d-1)$-complex.

A pure $d$-complex $C$ is \emph{shellable} if either (1) $C$ is a simplex, or (2) $C$ is a disjoint union of points, or (3) $d \ge 1$ and $C$ can be written as $C= C_1 \cup C_2$, where $C_1$ is a shellable  $d$-complex, $C_2$ is a $d$-simplex, and $C_1 \cap C_2$ is a shellable $(d-1)$-complex.

A pure $d$-complex $C$ is \emph{vertex-decomposable} if either (1) $C$ is a simplex, or (2) $C$ is a disjoint union of points, or (3) $d \ge 1$ and there is a vertex $v$ in $C$ (called \emph{shedding vertex}) such that $\del (v, C)$ and $\link (v, C)$ are both vertex-decomposable
(and $\del (v, C)$ is pure $d$-dimensional).

A (not necessarily pure!) $d$-complex $C$  is \emph{non-evasive} if either (1) $C$ is a simplex, or (2) $C$ is a single point, or (3) $d \ge 1$ and there is a vertex $v$ in $C$ such that $\del (v, C)$ and $\link (v, C)$ are both non-evasive.

An \emph{elementary collapse} is the simultaneous removal from a $d$-complex $C$ of a pair of faces $(\sigma, \Sigma)$ with the prerogative that $\Sigma$ is the only face properly containing $\sigma$.
(This condition is usually abbreviated in the expression 
`$\sigma$ is a free face of $\Sigma$'; some complexes have no free face).
If $C':= C - \Sigma - \sigma$, we say that the complex $C$
\emph{collapses onto} the complex $C'$. Even if $C$ is pure, this $C'$ need not be pure. We say that the complex $C$ \emph{collapses onto} $D$ if $C$ can be reduced to $D$ by some finite sequence of elementary collapses. A (not necessarily pure) $d$-complex $C$  is \emph{collapsible} if it collapses onto a single vertex.

A simplicial $3$-ball is \emph{locally constructible} (or shortly \emph{LC}) if it can be collapsed onto its boundary minus a triangle. A simplicial $3$-sphere is \emph{locally constructible} (or shortly \emph{LC}) if the removal of some tetrahedron makes it collapsible onto one of its vertices.

\subsection{Perfect discrete Morse functions}
A map $f: C \longrightarrow \mathbb{R}$ on a simplicial complex $C$ is a \emph{discrete Morse function on $C$} if for each face $\sigma$
\begin{compactenum}[(i)]
\item there is at most one boundary facet $\rho$ of $\sigma$ such that $f(\rho) \ge f(\sigma)$ and
\item there is at most one face $\tau$ having $\sigma$ as boundary facet such that $f(\tau) \le f(\sigma)$.
\end{compactenum}

\noindent A \emph{critical face} of $f$ is a face of $C$ for which
\begin{compactenum}[(i)]
\item there is no boundary facet $\rho$ of $\sigma$ such that $f(\rho) \ge f(\sigma)$ and
\item there is no face $\tau$ having $\sigma$ as boundary facet such that $f(\tau) \le f(\sigma)$.
\end{compactenum}

\noindent A \emph{collapse-pair} of $f$ is a pair of faces $(\sigma, \tau)$ such that
\begin{compactenum}[(i)]
\item $\sigma$ is a boundary facet of $\tau$ and
\item $f(\sigma) \ge f(\tau)$.  
\end{compactenum} 

Forman \cite[Section~2]{FormanADV} showed that for each discrete Morse function $f$ the collapse pairs of $f$ form a partial matching of the face poset of $C$. The unmatched faces are precisely the critical faces of~$f$. Each complex $K$ endowed with a discrete Morse function is homotopy equivalent to a cell complex with exactly one cell of dimension $i$ for each critical $i$-face \cite{FormanADV}. 
In particular, if we denote by $c_i (f)$ the number of critical $i$-faces of $f$, and by $\beta_i (C)$ the $i$-th Betti number of $C$, one has
\[ 
c_i (f) \ge \beta_i (C)
\]
for all discrete Morse functions $f$ on $C$. These inequalities need not be sharp.
If they are sharp for all $i$, the discrete Morse function is called \emph{perfect}.
However, for each $k$ and for each $d \ge 3$ there is a $d$-sphere $S$ \cite{Benedetti-DMT4MWB} such that for any discrete Morse function $f$ on $S$, one has
\[
c_{d-1} (f) \ge k + \beta_{d-1} (S) = k. 
\]

\subsection{Knots and knot-theoretic obstructions}

A \emph{knot} is a simple closed curve in a $3$-sphere. All the knots we consider are \emph{tame}, that is, realizable as $1$-dimensional subcomplexes of some triangulated $3$-sphere. A knot is \emph{trivial} if it bounds a disc; all the knots we consider here are non-trivial. The \emph{knot group} is the fundamental group of the knot complement inside the ambient sphere. For example, the knot group of the trefoil knot (and of its mirror image) is $\langle\, x,y \:\: | \:\: x^2 = y^3\,\rangle$. Ambient isotopic knots have isomorphic knot groups. A \emph{connected sum} of two knots is a knot obtained by cutting out a tiny arc from each and then sewing the resulting curves together along the boundary of the cutouts. For example, summing two trefoils one obtains the ``granny knot''; summing a trefoil and its mirror image one obtains the so-called ``square knot''. When we say ``double trefoil'', we mean any of these (granny knot or square knot): From the point of view of the knot group, it does not matter. A knot is 
\emph{$m$-complicated} if the knot group has a presentation with $m + 1$ generators, but no presentation with $m$ generators. By ``at least $m$-complicated'' we mean ``$k$-complicated for some $k \geq m$''. There exist arbitrarily complicated knots: Goodrick \cite{GOO} showed that the connected sum of $m$ trefoil knots is at least $m$-complicated.

A \emph{spanning edge} of a $3$-ball $B$ is an interior edge
that has both endpoints on the boundary $\partial B$. An \emph{$\mathfrak{L}$-knotted
spanning edge} of a $3$-ball $B$ is a spanning edge $xy$ such that some
simple path on $\partial B$ between $x$ and $y$ completes the edge to a
(non-trivial) knot $\mathfrak{L}$. From the simply-connectedness of $2$-spheres it follows that the knot type does not depend on the boundary path chosen; in other
words, the knot is determined by the edge. More generally, a \emph{spanning arc} is a path of interior edges in a $3$-ball $B$, such that both extremes of the path lie on the
boundary $\partial B$. If every path on $\partial B$ between the two endpoints of a spanning arc completes the latter to a knot $\mathfrak{L}$, the arc is called
\emph{$\mathfrak{L}$-knotted}. Note that the relative interior of the arc is allowed to intersect the boundary of the $3$-ball; compare Ehrenborg--Hachimori {\cite{EH}}.

Below is a list of known results on knotted spheres and balls. As for the notation, if $B$ is a $3$-ball with a knotted spanning edge, by $S_B$ we will mean the $3$-sphere $\partial (v \ast B)$, where $v$ is a new vertex. By $\mathfrak{L}_t$ we denote a connected sum of $t$ trefoil knots.

\begin{thm}[Benedetti/Ehrenborg/Hachimori/Ziegler] \label{thm:A}
Any $3$-ball with an \emph{$\mathfrak{L}_t$-knotted} spanning arc of $t$ edges cannot be LC \cite{Benedetti-DMT4MWB}, but it can be collapsible \cite{BZ, LICKMAR}. An arbitrary $3$-ball with an \emph{$\mathfrak{L}_1$-knotted} spanning arc of less than $3$ edges cannot be shellable nor constructible~\cite{HZ}. In contrast, some shellable $3$-balls have a \emph{$\mathfrak{L}_1$-knotted} spanning arc of $3$ edges \cite{HZ}.
\end{thm}

\begin{thm}[Adams et al. {\cite[Theorem~7.1]{AdamsEtAl}}] \label{thm:embed}
Any knotted $3$-ball in which the knot $\mathfrak{L}_t$ is realized with $e$ edges cannot be rectilinearly embeddable in $\mathbb{R}^3$ if\, $e \le 2t + 3$.
\end{thm}

\begin{thm}[Benedetti/Ehrenborg/Hachimori/Shimokawa/Ziegler] \label{thm:C}
A $3$-sphere or a $3$-ball, with a subcomplex of $m$ edges, isotopic to the sum of $t$ trefoil knots,
\begin{compactitem}[ --- ]
\item cannot be vertex-decomposable if\, $t \ge \lfloor \frac{m}{3} \rfloor$  \cite{HZ},
\item cannot be constructible/shellable if\, $t \ge \lfloor \frac{m}{2} \rfloor$ \cite{EH,HS}, and
\item cannot be LC if\, $t \ge m$ \cite{BZ}. 
\end{compactitem}
The first two bounds are known to be sharp for $t=1$ \cite{HZ}; the latter bound is sharp for all $t$, as far as spheres are concerned~\cite{Benedetti-diss,BZ}.
\end{thm}

\begin{thm} [Benedetti] \label{thm:D}
Let $S$ be a $3$-sphere with a subcomplex of $m$ edges, isotopic to the sum of $t$ trefoil knots. For \emph{any} discrete Morse function $f$ on $S$, one has
\[
c_{2} (f) \ge t - m + 1. 
\]  
\end{thm}

\section{The single trefoil} \label{sec:1}
In this section, we study the $3$-ball $B_{12,38}$ introduced in \cite{LUTZ1} and given by the following $38$ facets:
{\small
\[
\begin{array}{l@{\hspace{3.5mm}}l@{\hspace{3.5mm}}l@{\hspace{3.5mm}}l@{\hspace{3.5mm}}l@{\hspace{3.5mm}}l@{\hspace{3.5mm}}l@{\hspace{3.5mm}}l}
  2\,3\,4\,7,   & 2\,3\,4\,10,  & 2\,3\,7\,10,  & 2\,4\,5\,7,   & 
  2\,4\,5\,10,  & 2\,5\,7\,13,  & 2\,5\,8\,10,  & 2\,5\,8\,13,  \\

  2\,6\,9\,11,  & 2\,6\,11\,13, & 2\,6\,12\,13, & 2\,7\,8\,10,  & 
  2\,7\,8\,11,  & 2\,7\,11\,13, & 2\,8\,9\,11,  & 2\,8\,9\,12,  \\

  2\,8\,12\,13, & 3\,4\,6\,7,   & 3\,4\,6\,10,  & 3\,5\,8\,13,  &
  3\,5\,9\,11,  & 3\,5\,9\,13,  & 3\,6\,7\,12,  & 3\,6\,10\,13, \\

  3\,6\,12\,13, & 3\,7\,10\,12, & 3\,8\,9\,11,  & 3\,8\,9\,12,  &
  3\,8\,12\,13, & 3\,9\,10\,12, & 3\,9\,10\,13, & 4\,5\,6\,7,   \\

  4\,5\,6\,10,  & 5\,6\,7\,9,   & 5\,6\,9\,11,  & 5\,6\,10\,11, & 
  5\,7\,9\,13,  & 6\,10\,11\,13.
\end{array}
\]
}%
The ball is contructed in a way such that the edge $2\,3$ is a knotted spanning edge for $B_{12,38}$, the knot being a single trefoil. In particular, by Theorem~\ref{thm:A}, $B_{12,38}$ is not shellable, not constructible and not LC. Here we show that:
\begin{compactenum}[(1)]
\item $B_{12,38}$ is not rectilinearly-embeddable in $\mathbb{R}^3$;
\item $B_{12,38}$ is evasive;
\item $B_{12,38}$ is collapsible;
\item The $3$-sphere $\partial (1 \ast B_{12,38})$ minus 
the facet $1\,2\,6\,9$ is an LC knotted $3$-ball.
\end{compactenum}

\pagebreak

\begin{prop}
$B_{12,38}$ is not rectilinearly-embeddable in $\mathbb{R}^3$. 
\end{prop}

\begin{proof}
The boundary of $B_{12,38}$ consists of the following $18$ triangles:
{\small
\[
\begin{array}{l@{\hspace{3.5mm}}l@{\hspace{3.5mm}}l@{\hspace{3.5mm}}l@{\hspace{3.5mm}}l@{\hspace{3.5mm}}l@{\hspace{3.5mm}}l@{\hspace{3.5mm}}l@{\hspace{3.5mm}}l}
  2\,6\,9,   & 2\,6\,12,  & 2\,9\,12,  & 3\,5\,8 ,    & 3\,5\,11, & 
  3\,8\,11,  & 5\,8\,10,  & 5\,10\,11, & 6\,7\,9,     \\

  6\,7\,12,  & 7\,8\,10 , & 7\,8\,11,  & 7\,9\,13,    & 7\,10\,12,  &
  7\,11\,13, & 9\,10\,12, & 9\,10\,13, & 10\,11\,13. \\
\end{array}
\]
}%
In particular, the four edges $2\,6$, $6\,7$, $7\,8$ and $3\,8$ 
form a boundary path from the vertex $2$ to the vertex $3$.
Together with the interior edge $2\,3$, this path closes up to 
a pentagonal trefoil knot. By Theorem~\ref{thm:embed}, 
$B_{12,38}$ cannot be rectilinearly embedded in $\mathbb{R}^3$,
because the stick number of the trefoil knot is $6$.
\end{proof}

\begin{prop} \label{prop:B12collapsible}
$B_{12,38}$ is collapsible, but not LC.
\end{prop}
 
\begin{proof}
By Theorem~\ref{thm:A}, $B$ is not LC; in particular, 
$B$ does not collapse onto its boundary minus a triangle. 
So, in the first phase of the collapse (the one in which 
the tetrahedra are collapsed away) we have to remove 
several boundary triangles in order to succeed. Now, 
\emph{finding} a collapse can be difficult, but \emph{verifying} 
the correctness of a given collapse is fast. 
The following is a certificate of the collapsibility of $B_{12,38}$.

\vspace{3mm}

\noindent {\footnotesize First phase (pairs ``triangle'' $\rightarrow$ ``tetrahedron''):
\setlength{\arraycolsep}{0.1em}
\[\!\!
\begin{array}{r@{\hspace{2mm}}l@{\hspace{2mm}}l@{\hspace{4mm}}r@{\hspace{2mm}}l@{\hspace{2mm}}l@{\hspace{4mm}}r@{\hspace{2mm}}l@{\hspace{2mm}}l@{\hspace{4mm}}r@{\hspace{2mm}}l@{\hspace{2mm}}l@{\hspace{4mm}}r@{\hspace{2mm}}l@{\hspace{2mm}}l}
 10\,11\,13 & \rightarrow & 6\,10\,11\,13, &
 7\,9\,13   & \rightarrow & 5\,7\,9\,13,   &
 6\,10\,11  & \rightarrow & 5\,6\,10\,11,  &
 5\,6\,11   & \rightarrow & 5\,6\,9\,11,   &
 2\,6\,12   & \rightarrow & 2\,6\,12\,13,  \\

 5\,7\,9    & \rightarrow & 5\,6\,7\,9,    &
 9\,10\,12  & \rightarrow & 3\,9\,10\,12,  &
 7\,11\,13  & \rightarrow & 2\,7\,11\,13,  &
 5\,9\,11   & \rightarrow & 3\,5\,9\,11,   &
 2\,7\,13   & \rightarrow & 2\,5\,7\,13,   \\

 3\,9\,12   & \rightarrow & 3\,8\,9\,12,   &
 2\,6\,13   & \rightarrow & 2\,6\,11\,13,  &
 3\,8\,12   & \rightarrow & 3\,8\,12\,13,  &
 3\,9\,11   & \rightarrow & 3\,8\,9\,11,   &
 7\,10\,12  & \rightarrow & 3\,7\,10\,12,  \\

 8\,9\,12   & \rightarrow & 2\,8\,9\,12,   &
 6\,10\,13  & \rightarrow & 3\,6\,10\,13,  &
 3\,5\,8    & \rightarrow & 3\,5\,8\,13,   &
 6\,9\,11   & \rightarrow & 2\,6\,9\,11,   &
 8\,12\,13  & \rightarrow & 2\,8\,12\,13,  \\

 3\,6\,13   & \rightarrow & 3\,6\,12\,13,  &
 3\,10\,13  & \rightarrow & 3\,9\,10\,13,  &
 3\,5\,13   & \rightarrow & 3\,5\,9\,13,   &
 6\,7\,12   & \rightarrow & 3\,6\,7\,12,   &
 3\,6\,7    & \rightarrow & 3\,4\,6\,7,    \\

 5\,6\,7    & \rightarrow & 4\,5\,6\,7,    &
 7\,8\,11   & \rightarrow & 2\,7\,8\,11,   &
 2\,9\,11   & \rightarrow & 2\,8\,9\,11,   &
 3\,4\,6    & \rightarrow & 3\,4\,6\,10,   &
 4\,5\,7    & \rightarrow & 2\,4\,5\,7,    \\

 5\,6\,10   & \rightarrow & 4\,5\,6\,10,   &
 3\,4\,10   & \rightarrow & 2\,3\,4\,10,   &
 2\,4\,7    & \rightarrow & 2\,3\,4\,7,    &
 2\,3\,7    & \rightarrow & 2\,3\,7\,10,   &
 5\,8\,10   & \rightarrow & 2\,5\,8\,10,   \\

 5\,8\,13   & \rightarrow & 2\,5\,8\,13,   &
 7\,8\,10   & \rightarrow & 2\,7\,8\,10,   &
 2\,4\,5    & \rightarrow & 2\,4\,5\,10.   &
\end{array}
\]

\noindent Second phase (pairs ``edge'' $\rightarrow$ ``triangle''):
\setlength{\arraycolsep}{0.15em}
\[\!\!\!\!\!\!\!
\begin{array}{r@{\hspace{2mm}}l@{\hspace{2mm}}l@{\hspace{4mm}}r@{\hspace{2mm}}l@{\hspace{2mm}}l@{\hspace{4mm}}r@{\hspace{2mm}}l@{\hspace{2mm}}l@{\hspace{4mm}}r@{\hspace{2mm}}l@{\hspace{2mm}}l@{\hspace{4mm}}r@{\hspace{2mm}}l@{\hspace{2mm}}l@{\hspace{4mm}}r@{\hspace{2mm}}l@{\hspace{2mm}}l}
 8\,12  & \rightarrow & 2\,8\,12,  &
 7\,8   & \rightarrow & 2\,7\,8,   &
 7\,13  & \rightarrow & 5\,7\,13,  &
 8\,10  & \rightarrow & 2\,8\,10,  &
 9\,11  & \rightarrow & 8\,9\,11,  &
 7\,9   & \rightarrow & 6\,7\,9,   \\

 10\,11 & \rightarrow & 5\,10\,11, &
 7\,11  & \rightarrow & 2\,7\,11,  &
 5\,8   & \rightarrow & 2\,5\,8,   &
 9\,12  & \rightarrow & 2\,9\,12,  &
 7\,12  & \rightarrow & 3\,7\,12,  &
 5\,11  & \rightarrow & 3\,5\,11,  \\

 3\,5   & \rightarrow & 3\,5\,9,   &
 5\,7   & \rightarrow & 2\,5\,7,   &
 10\,12 & \rightarrow & 3\,10\,12, &
 3\,11  & \rightarrow & 3\,8\,11,  &
 6\,7   & \rightarrow & 4\,6\,7,   &
 4\,7   & \rightarrow & 3\,4\,7,   \\

 2\,7   & \rightarrow & 2\,7\,10,  &
 8\,11  & \rightarrow & 2\,8\,11,  &
 2\,12  & \rightarrow & 2\,12\,13, &
 10\,13 & \rightarrow & 9\,10\,13, &
 3\,4   & \rightarrow & 2\,3\,4,   &
 2\,3   & \rightarrow & 2\,3\,10,  \\

 7\,10  & \rightarrow & 3\,7\,10,  &
 9\,10  & \rightarrow & 3\,9\,10,  &
 3\,10  & \rightarrow & 3\,6\,10,  &
 6\,10  & \rightarrow & 4\,6\,10,  &
 4\,6   & \rightarrow & 4\,5\,6,   &
 4\,5   & \rightarrow & 4\,5\,10,  \\

 2\,4   & \rightarrow & 2\,4\,10,  &
 3\,6   & \rightarrow & 3\,6\,12,  &
 2\,10  & \rightarrow & 2\,5\,10,  &
 3\,12  & \rightarrow & 3\,12\,13, &
 12\,13 & \rightarrow & 6\,12\,13, &
 2\,5   & \rightarrow & 2\,5\,13,  \\

 5\,6   & \rightarrow & 5\,6\,9,   &
 6\,13  & \rightarrow & 6\,11\,13, &
 5\,13  & \rightarrow & 5\,9\,13,  &
 11\,13 & \rightarrow & 2\,11\,13, &
 2\,13  & \rightarrow & 2\,8\,13,  &
 9\,13  & \rightarrow & 3\,9\,13,  \\

 6\,9   & \rightarrow & 2\,6\,9,   &
 3\,9   & \rightarrow & 3\,8\,9,   &
 3\,8   & \rightarrow & 3\,8\,13,  &
 2\,8   & \rightarrow & 2\,8\,9,   &
 6\,11  & \rightarrow & 2\,6\,11.  &
\end{array}
\]

\noindent Third phase (pairs ``vertex'' $\rightarrow$ ``edge''):
\setlength{\arraycolsep}{0.22em}
\[\!\!\!\!\!\!\!\!\!\!\!\!
\begin{array}{rl@{\hspace{2mm}}l@{\hspace{3mm}}r@{\hspace{2mm}}l@{\hspace{2mm}}l@{\hspace{3mm}}r@{\hspace{2mm}}l@{\hspace{2mm}}l@{\hspace{3mm}}r@{\hspace{2mm}}l@{\hspace{2mm}}l@{\hspace{3mm}}r@{\hspace{2mm}}l@{\hspace{2mm}}l@{\hspace{3mm}}r@{\hspace{2mm}}l@{\hspace{2mm}}l@{\hspace{3mm}}r@{\hspace{2mm}}l@{\hspace{2mm}}l@{\hspace{3mm}}r@{\hspace{2mm}}l@{\hspace{2mm}}l@{\hspace{3mm}}r@{\hspace{2mm}}l@{\hspace{2mm}}l}
 12 & \rightarrow & 6\,12, &
 4  & \rightarrow & 4\,10, &
 6  & \rightarrow & 2\,6,  &
 10 & \rightarrow & 5\,10, &
 11 & \rightarrow & 2\,11, &
 5  & \rightarrow & 5\,9,  &
 7  & \rightarrow & 3\,7,  &
 2  & \rightarrow & 2\,9,  &
 9  & \rightarrow & 8\,9,  \\

 3  & \rightarrow & 3\,13, &
 13 & \rightarrow & 8\,13. &
\end{array}
\]
}
\end{proof}

The above collapsing sequence was found with the randomized approach of \cite{BenedettiLutz2012bpre}.

\begin{prop} \label{prop:B12evasive}
$B_{12,38}$ is evasive. 
\end{prop}

\begin{proof} Let us establish some notation first. We identify each vertex of $B_{12,38}$ with its label, which is an integer in $A:=\{2, \ldots, 13\}$. For each subset $S$ of $A$, we denote by $C_S$ the complex obtained from $B_{12,38}$ by deleting the vertices in $S$. 

Now, suppose by contradiction that $B$ is non-evasive. The vertices of $B_{12,38}$ can be reordered so that their progressive deletions and links are non-evasive. In particular, there exists a five-element subset $F$ of $A$ such that $C_F$ is non-evasive.  

With the help of a computer program, we checked the homologies of all complexes obtained by deleting five vertices from $B$. Since the order of deletion does not matter, there are only $\binom{12}{5} = 792$ cases to check, so the computation is extremely fast. It turns out that these homologies are never trivial, except for the following three cases:
\begin{compactenum}[(1)]
\item $F_1 = \{4,5,8,10,11\}$, 
\item $F_2 = \{4,5,10,11,12\}$,
\item $F_3 = \{4,6,7,9,12\}$.
\end{compactenum}
So, the non-evasive complex $C_F$ whose existence was postulated above must be either $C_{F_1}$, or $C_{F_2}$, or $C_{F_3}$.
However, it is easy to see that the deletion of any vertex from $C_{F_1}$ yields a non-acyclic complex. The same holds for $C_{F_2}$ and $C_{F_3}$. Therefore, all three complexes $C_{F_1}$, $C_{F_2}$ and $C_{F_3}$ are evasive: A~contradiction.
\end{proof}

\begin{remark}
Let $S_B$ be the sphere obtained by coning off the boundary of $B_{12,38}$ with an extra vertex, 
labeled by~$1$. Let $\Sigma$ be the tetrahedron $1\,2\,6\,9$ and let $\sigma$ 
be its facet $2\,6\,9$.
With the help of the computer, one can check that $S_B - \Sigma$ collapses 
onto the $2$-ball $D$ consisting of the triangles $1\,2\,6$, $1\,2\,9$ and $1\,6\,9$. 
Since $D = \partial \Sigma - \sigma = \partial (S_B - \Sigma) - \sigma$, 
it follows that the knotted $3$-ball $S_B - \Sigma$ is locally constructible 
(because it collapses onto its boundary minus the triangle $\sigma$). 
For a proof, see \cite{Benedetti-diss}.
\end{remark}

\section{The double trefoil}
\label{sec:2}

  \begin{figure}
    \begin{center}
      \small
      \psfrag{1}{1}
      \psfrag{2}{2}
      \psfrag{3}{3}
      \psfrag{4}{4}
      \psfrag{5}{5}
      \psfrag{6}{6}
      \psfrag{7}{7}
      \psfrag{8}{8}
      \psfrag{9}{9}
      \psfrag{10}{10}
      \psfrag{11}{11}
      \psfrag{12}{12}
      \psfrag{13}{13}
      \psfrag{14}{14}
      \psfrag{15}{15}
      \psfrag{16}{16}
      \psfrag{17}{17}
      \psfrag{18}{18}
      \psfrag{19}{19}
      \psfrag{20}{20}
      \psfrag{21}{21}
      \psfrag{22}{22}
      \psfrag{23}{23}
      \psfrag{24}{24}
      \psfrag{25}{25}
      \psfrag{26}{26}
      \psfrag{27}{27}
      \psfrag{28}{28}
      \psfrag{29}{29}
      \psfrag{30}{30}
      \psfrag{31}{31}
      \psfrag{32}{32}
      \includegraphics[width=11.5cm]{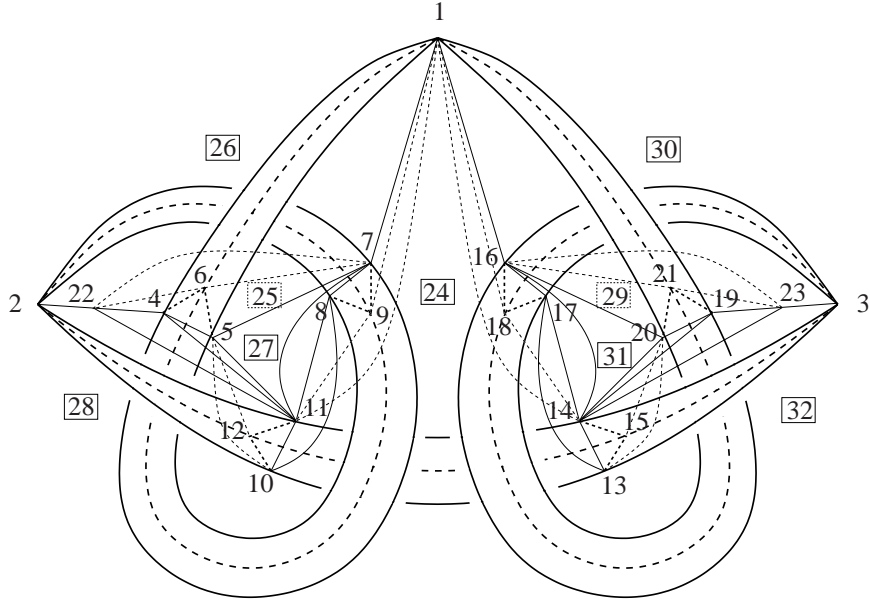}
    \end{center}
    \caption{The double trefoil in the sphere $S_{33,192}$.
      \label{fig:double_trefoil}}
  \end{figure}

In the following, we present the construction of a triangulated $3$-sphere  that contains a double trefoil knot 
on three edges in its $1$-skeleton. In fact, there are two different ways to form the connected sum of
two trefoil knots,  the granny and the square knot. We base our construction on the square knot. 

Let $1\,2$,  $2\,3$, $1\,3$ be the three edges forming the square knot, which, for our purposes, we simply call the \emph{double trefoil knot}. 
An embedding of the knot in ${\mathbb R}^3$ is depicted in Figure~\ref{fig:double_trefoil}.

Our strategy to place the knot into the $1$-skeleton of a triangulated $3$-dimensional sphere
is as follows. We 
\begin{compactitem}
\item start with an embedding of the knot in ${\mathbb R}^3$,
\item triangulate the region around the knot to obtain a triangulated $3$-ball,
\item complete it to a triangulation of $S^3$ by adding the cone over its boundary.
\end{compactitem}
Once the knot edges $1\,2$,  $2\,3$, $1\,3$ are placed in ${\mathbb R}^3$
we need to shield off these edges to prevent unwanted identifications of distant vertices
later on. We protect each of the knot edges by placing a spindle around it;
see Figure~\ref{fig:spindles_double_trefoil} for images of the spindles and
Table~\ref{tbl:spindles_double_trefoil} for lists of nine tetrahedra each, which form the three spindles.
The additional vertices on the boundaries of the spindles allow us to close the holes 
of the knot by gluing in (triangulated) membrane patches.

  \begin{figure}
    \begin{center}
      \small
      \psfrag{1}{1}
      \psfrag{2}{2}
      \psfrag{3}{3}
      \psfrag{4}{4}
      \psfrag{5}{5}
      \psfrag{6}{6}
      \psfrag{7}{7}
      \psfrag{8}{8}
      \psfrag{9}{9}
      \psfrag{10}{10}
      \psfrag{11}{11}
      \psfrag{12}{12}
      \psfrag{13}{13}
      \psfrag{14}{14}
      \psfrag{15}{15}
      \psfrag{16}{16}
      \psfrag{17}{17}
      \psfrag{18}{18}
      \psfrag{19}{19}
      \psfrag{20}{20}
      \psfrag{21}{21}
      \includegraphics[width=9.6cm]{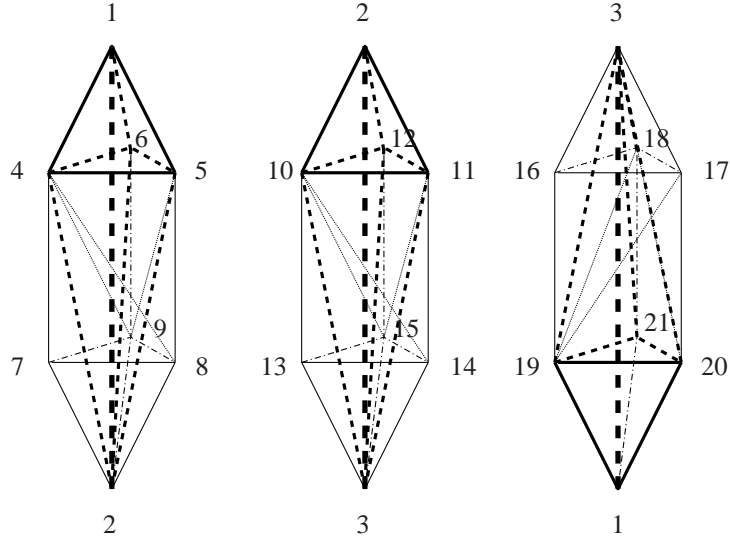}
    \end{center}
    \caption{The spindles of $S_{33,192}$.
      \label{fig:spindles_double_trefoil}}
  \end{figure}

\begin{table}[htp]
\small\centering
\defaultaddspace=0.6em
\caption{Part I of the sphere $S_{33,192}$: The three spindles.}\label{tbl:spindles_double_trefoil}
\begin{center}\footnotesize
\begin{tabular*}{\linewidth}{@{\extracolsep{\fill}}ll@{\hspace{30mm}}ll@{\hspace{30mm}}ll@{}}
\toprule
 \addlinespace
$1\,2\,4\,5$ & $2\,4\,7\,8$   &   $2\,3\,10\,11$ & $3\,10\,13\,14$   &   $1\,3\,19\,20$ & $3\,16\,17\,19$ \\
$1\,2\,4\,6$ & $2\,4\,5\,8$   &   $2\,3\,10\,12$ & $3\,10\,11\,14$   &   $1\,3\,19\,21$ & $3\,17\,19\,20$ \\
$1\,2\,5\,6$ & $2\,5\,8\,9$   &   $2\,3\,11\,12$ & $3\,11\,14\,15$   &   $1\,3\,20\,21$ & $3\,17\,18\,20$ \\
             & $2\,5\,6\,9$   &                  & $3\,11\,12\,15$   &                  & $3\,18\,20\,21$ \\
             & $2\,4\,7\,9$   &                  & $3\,10\,13\,15$   &                  & $3\,16\,18\,19$ \\
             & $2\,4\,6\,9$   &                  & $3\,10\,12\,15$   &                  & $3\,18\,19\,21$ \\
 \addlinespace
\bottomrule
\end{tabular*}
\end{center}
\end{table}

In Figure~\ref{fig:double_trefoil}, the diagonal edges on the boundaries of the spindles and also 
the interior edges of the spindles are not shown. All that we need at the moment are the vertices
on the boundaries of the spindles. For example, if we move along the left spindle $1$--$2$
from apex $1$ to apex $2$, we first meet the vertices $4$, $5$, $6$ 
and then the vertices $7$, $8$, $9$ on the spindle boundary. 

\begin{table}[t]
\small\centering
\defaultaddspace=0.6em
\caption{The triangles of the membranes in the sphere $S_{33,192}$.}\label{tbl:membranes_double_trefoil}
\begin{center}\footnotesize
\begin{tabular*}{\linewidth}{@{\extracolsep{\fill}}lllllll@{}}
\toprule
 \addlinespace
             &               &              &  $1\,11\,14$  &                &                &                \\
$4\,5\,11$   &  $1\,5\,7$    &  $1\,9\,11$  &               &  $1\,14\,18$   &  $1\,16\,20$   &  $14\,19\,20$  \\
$4\,11\,22$  &  $5\,7\,11$   &  $1\,7\,9$   &               &  $1\,16\,18$   &  $14\,16\,20$  &  $14\,19\,23$  \\
$2\,11\,22$  &  $7\,8\,11$   &  $8\,9\,11$  &               &  $14\,17\,18$  &  $14\,16\,17$  &  $3\,14\,23$   \\
$2\,7\,22$   &  $8\,10\,11$  &              &               &                &  $13\,14\,17$  &  $3\,16\,23$   \\
$6\,7\,22$   &  $5\,8\,10$   &              &               &                &  $13\,17\,20$  &  $16\,21\,23$  \\
$4\,6\,22$   &  $5\,10\,12$  &              &               &                &  $13\,15\,20$  &  $19\,21\,23$  \\
$5\,6\,7$    &  $5\,11\,12$  &              &               &                &  $14\,15\,20$  &  $16\,20\,21$  \\
 \addlinespace
\bottomrule
\end{tabular*}
\end{center}
\end{table}

The membrane patches can be read off from Table~\ref{tbl:membranes_double_trefoil}. The central triangle $1\,11\,14$
connects the left part with the right part of Figure~\ref{fig:double_trefoil} and contributes to the closure 
of the upper central hole. Next to the triangle $1\,11\,14$ on the left hand side in Figure~\ref{fig:double_trefoil}
is the triangle $1\,9\,11$ from the third column of Table~\ref{tbl:membranes_double_trefoil}, followed by triangle $1\,7\,9$  and so on.
Once all the membrane triangles of Table~\ref{tbl:membranes_double_trefoil} are in place in Figure~\ref{fig:double_trefoil}, 
the resulting complex is a mixed $2$- and $3$-dimensional simplicial complex,
consisting of spindle tetrahedra and membrane triangles. Since we closed all holes
of the initial double trefoil knot, the resulting complex is contractible.

\begin{table}[htp]
\small\centering
\defaultaddspace=0.6em
\caption{Part II of the sphere $S_{33,192}$: Tetrahedra to be added to Part I to obtain a ball $B_{32,140}$.}\label{tbl:local_cones}
\begin{center}\footnotesize
\begin{tabular*}{\linewidth}{@{\extracolsep{\fill}}lll@{\hspace{14mm}}l@{\hspace{14mm}}lll@{}}
\toprule
 \addlinespace
$4\,6\,24\,25$  & $1\,7\,9\,26$   & $1\,5\,26\,27$   & $1\,7\,9\,24$    & $19\,21\,24\,29$ & $1\,16\,18\,30$  & $1\,20\,30\,31$  \\
$5\,6\,24\,25$  & $2\,7\,9\,26$   & $5\,11\,26\,27$  & $1\,9\,11\,24$   & $20\,21\,24\,29$ & $3\,16\,18\,30$  & $14\,20\,30\,31$ \\
$5\,10\,24\,25$ & $1\,5\,6\,7$    & $10\,11\,26\,27$ & $8\,9\,11\,24$   & $13\,20\,24\,29$ & $1\,16\,20\,21$  & $13\,14\,30\,31$ \\
$5\,10\,12\,25$ & $1\,6\,7\,26$   & $8\,10\,11\,27$  & $8\,10\,11\,24$  & $13\,15\,20\,29$ & $1\,16\,21\,30$  & $13\,14\,17\,31$ \\
$5\,11\,12\,25$ & $6\,7\,22\,26$  & $7\,8\,11\,27$   & $5\,8\,10\,24$   & $14\,15\,20\,29$ & $16\,21\,23\,30$ & $14\,16\,17\,31$ \\
$5\,7\,11\,25$  & $2\,7\,22\,26$  & $5\,7\,11\,27$   & $5\,6\,9\,24$    & $14\,16\,20\,29$ & $3\,16\,23\,30$  & $14\,16\,20\,31$ \\
$7\,8\,11\,25$  & $1\,4\,6\,26$   & $1\,5\,7\,27$    & $5\,8\,9\,24$    & $14\,16\,17\,29$ & $1\,19\,21\,30$  & $1\,16\,20\,31$  \\
$8\,9\,11\,25$  & $4\,6\,22\,26$  &                  & $4\,6\,9\,24$    & $14\,17\,18\,29$ & $19\,21\,23\,30$ &                  \\
$5\,6\,7\,25$   & $1\,4\,5\,26$   & $8\,10\,27\,28$  & $4\,7\,9\,24$    & $16\,20\,21\,29$ & $1\,19\,20\,30$  & $13\,17\,31\,32$ \\
$6\,7\,22\,25$  & $4\,5\,11\,26$  & $5\,8\,10\,28$   &                  & $16\,21\,23\,29$ & $14\,19\,20\,30$ & $13\,17\,20\,32$ \\
$4\,6\,22\,25$  & $4\,11\,22\,26$ & $5\,10\,12\,28$  & $1\,11\,14\,24$  & $19\,21\,23\,29$ & $14\,19\,23\,30$ & $13\,15\,20\,32$ \\
$2\,7\,22\,25$  & $2\,11\,22\,26$ & $5\,11\,12\,28$  & $10\,11\,14\,24$ & $3\,16\,23\,29$  & $3\,14\,23\,30$  & $14\,15\,20\,32$ \\
$2\,7\,8\,25$   & $2\,10\,11\,26$ & $4\,5\,11\,28$   & $10\,13\,14\,24$ & $3\,16\,17\,29$  & $3\,13\,14\,30$  & $14\,19\,20\,32$ \\
$2\,8\,9\,25$   &                 & $4\,11\,22\,28$  &                  & $3\,17\,18\,29$  &                  & $14\,19\,23\,32$ \\
                &                 & $2\,11\,22\,28$  & $1\,16\,18\,24$  &                  &                  & $3\,14\,23\,32$  \\
                &                 & $2\,11\,12\,28$  & $1\,14\,18\,24$  &                  &                  & $3\,14\,15\,32$  \\
                &                 & $2\,10\,12\,28$  & $14\,17\,18\,24$ &                  &                  & $3\,13\,15\,32$  \\
                &                 & $2\,10\,26\,28$  & $13\,14\,17\,24$ &                  &                  & $3\,13\,30\,32$  \\
                &                 & $10\,26\,27\,28$ & $13\,17\,20\,24$ &                  &                  & $13\,30\,31\,32$ \\
                &                 & $4\,5\,8\,28$    & $17\,18\,20\,24$ &                  &                  & $17\,19\,20\,32$ \\
                &                 &                  & $18\,20\,21\,24$ &                  &                  &                  \\
                &                 &                  & $16\,18\,19\,24$ &                  &                  &                  \\
                &                 &                  & $18\,19\,21\,24$ &                  &                  &                  \\
 \addlinespace
\bottomrule
\end{tabular*}
\end{center}
\end{table}

\pagebreak

Our next aim is to thicken the intermediate mixed $2$- and $3$-dimensional complex
to a triangulated $3$-ball $B_{32,140}$. For this end we add local cones to Figure~\ref{fig:double_trefoil}
with respect to the nine new vertices $24$, $25$, \dots, $32$. These cones are listed in Table~\ref{tbl:local_cones},
the positions of their apices are marked in Figure~\ref{fig:double_trefoil} by boxes containing the
new vertices. 

\begin{table}[t]
\small\centering
\defaultaddspace=0.6em
\caption{Part III of the sphere $S_{33,192}$: Cone over the boundary of the ball $B_{32,140}$.}\label{tbl:cone_over_boundary}
\begin{center}\footnotesize
\begin{tabular*}{\linewidth}{@{\extracolsep{\fill}}lllllll@{}}
\toprule
 \addlinespace
  $1\,7\,24\,33$   & $1\,7\,27\,33$   & $1\,9\,11\,33$   & $1\,9\,26\,33$   & $1\,11\,14\,33$  & $1\,14\,18\,33$  & $1\,16\,24\,33$  \\
  $1\,16\,31\,33$  & $1\,18\,30\,33$  & $1\,26\,27\,33$  & $1\,30\,31\,33$  & $2\,9\,25\,33$   & $2\,9\,26\,33$   & $2\,22\,25\,33$  \\ 
  $2\,22\,28\,33$  & $2\,26\,28\,33$  & $3\,18\,29\,33$  & $3\,18\,30\,33$  & $ 3\,23\,29\,33$ & $3\,23\,32\,33$  & $3\,30\,32\,33$  \\
  $4\,7\,8\,33$    & $4\,7\,24\,33$   & $4\,8\,28\,33$   & $4\,22\,25\,33$  & $4\,22\,28\,33$  & $4\,24\,25\,33$  & $7\,8\,27\,33$   \\ 
  $8\,27\,28\,33$  & $9\,11\,25\,33$  & $10\,12\,15\,33$ & $10\,12\,25\,33$ & $10\,13\,15\,33$ & $10\,13\,24\,33$ & $10\,24\,25\,33$ \\ 
  $11\,12\,15\,33$ & $11\,12\,25\,33$ & $11\,14\,15\,33$ & $13\,15\,29\,33$ & $13\,24\,29\,33$ & $14\,15\,29\,33$ & $14\,18\,29\,33$ \\
  $16\,17\,19\,33$ & $16\,17\,31\,33$ & $16\,19\,24\,33$ & $17\,19\,32\,33$ & $17\,31\,32\,33$ & $19\,23\,29\,33$ & $19\,23\,32\,33$ \\
  $19\,24\,29\,33$ & $26\,27\,28\,33$ & $30\,31\,32\,33$ &                  &                  &                  &                  \\
 \addlinespace
\bottomrule
\end{tabular*}
\end{center}
\end{table}

If we add together all the (spindle) tetrahedra from Table~\ref{tbl:spindles_double_trefoil} (Part I of the sphere $S_{33,192}$)
with all the (cone) tetrahedra from Table~\ref{tbl:local_cones} (Part II of the sphere $S_{33,192}$),
we obtain a triangulated $3$-ball $B_{32,140}$ with $32$ vertices and $140$ tetrahedral facets.
By construction, the $3$-ball $B_{32,140}$ contains the double trefoil knot in its $1$-skeleton
and all the membrane triangles in its $2$-skeleton.

\enlargethispage*{9mm}

In a final step, we add to the $3$-ball $B_{32,140}$ the cone over its boundary with respect the vertex $33$ 
(Part III of the sphere $S_{33,192}$ with tetrahedra as listed in Table~\ref{tbl:cone_over_boundary})
to obtain the $3$-sphere $S_{33,192}$.

\begin{table}[h]
\small\centering
\defaultaddspace=0.6em
\caption{The sphere $S_{16,92}$.}\label{tbl:double_trefoil_sphere}
\begin{center}\footnotesize
\begin{tabular*}{\linewidth}{@{\extracolsep{\fill}}llllllll@{}}
\toprule
 \addlinespace
  $1\,2\,5\,6$    & $1\,2\,5\,12$   & $1\,2\,6\,12$   & $1\,3\,7\,8$     & $1\,3\,7\,11$   & $1\,3\,8\,11$   & $1\,4\,5\,6$    & $1\,4\,5\,16$   \\
  $1\,4\,6\,12$   & $1\,4\,10\,13$  & $1\,4\,10\,16$  & $1\,4\,12\,13$   & $1\,5\,12\,13$  & $1\,5\,13\,16$  & $1\,7\,8\,9$    & $1\,7\,9\,11$   \\
  $1\,8\,9\,14$   & $1\,8\,10\,14$  & $1\,8\,10\,15$  & $1\,8\,11\,15$   & $1\,9\,11\,15$  & $1\,9\,14\,15$  & $1\,10\,13\,14$ & $1\,10\,15\,16$ \\ 
  $1\,13\,14\,16$ & $1\,14\,15\,16$ & $2\,3\,4\,13$   & $2\,3\,4\,15$    & $2\,3\,13\,15$  & $2\,4\,7\,8$    & $2\,4\,7\,15$   & $2\,4\,8\,16$   \\
  $2\,4\,10\,13$  & $2\,4\,10\,16$  & $2\,5\,6\,14$   & $2\,5\,12\,14$   & $2\,6\,8\,12$   & $2\,6\,8\,16$   & $2\,6\,9\,14$   & $2\,6\,9\,16$   \\
  $2\,7\,8\,9$    & $2\,7\,9\,10$   & $2\,7\,10\,13$  & $2\,7\,13\,15$   & $2\,8\,9\,14$   & $2\,8\,12\,14$  & $2\,9\,10\,16$  & $3\,4\,12\,13$  \\
  $3\,4\,12\,15$  & $3\,5\,6\,7$    & $3\,5\,6\,14$   & $3\,5\,7\,8$     & $3\,5\,8\,11$   & $3\,5\,11\,14$  & $3\,6\,7\,16$   & $3\,6\,9\,14$   \\
  $3\,6\,9\,16$   & $3\,7\,11\,14$  & $3\,7\,14\,16$  & $3\,9\,12\,13$   & $3\,9\,12\,16$  & $3\,9\,13\,15$  & $3\,9\,14\,15$  & $3\,12\,15\,16$ \\
  $3\,14\,15\,16$ & $4\,5\,6\,7$    & $4\,5\,7\,8$    & $4\,5\,8\,16$    & $4\,6\,7\,15$   & $4\,6\,12\,15$  & $5\,8\,11\,13$  & $5\,8\,13\,16$  \\
  $5\,11\,12\,13$ & $5\,11\,12\,14$ & $6\,7\,13\,15$  & $6\,7\,13\,16$   & $6\,8\,12\,15$  & $6\,8\,13\,15$  & $6\,8\,13\,16$  & $7\,9\,10\,12$  \\
  $7\,9\,11\,12$  & $7\,10\,12\,14$ & $7\,10\,13\,14$ & $7\,11\,12\,14$  & $7\,13\,14\,16$ & $8\,10\,12\,14$ & $8\,10\,12\,15$ & $8\,11\,13\,15$ \\
  $9\,10\,12\,16$ & $9\,11\,12\,13$ & $9\,11\,13\,15$ & $10\,12\,15\,16$ &                 &                 &                 &                 \\
 \addlinespace
\bottomrule
\end{tabular*}
\end{center}
\end{table}

\begin{prop}
The $3$-sphere $S_{33,192}$ consists of $192$ tetrahedra and $33$ vertices.
It has face vector $f=(33,225,384,192)$ 
and contains the double trefoil knot on three edges in its $1$-skeleton.
\end{prop}

The $3$-sphere $S_{33,192}$ is not minimal with the property of containing the double trefoil knot
in its $1$-skeleton. One way of obtaining smaller triangulations is by applying bistellar flips, cf. \cite{BjoernerLutz2000},
to the triangulation $S_{33,192}$. If we want to keep the knot while doing local bistellar modifications
on the triangulation, we merely have to exclude the knot edges $1\,2$,  $2\,3$, $1\,3$ as pivot edges
in the bistellar flip program BISTELLAR \cite{Lutz_BISTELLAR}. The smallest triangulation we found this way
is $S_{16,92}$; see Table~\ref{tbl:double_trefoil_sphere} for the list of facets of $S_{16,92}$.

\begin{thm}
The $3$-sphere $S_{16,92}$ has $92$ tetrahedra and $16$ vertices.
It has face vector \linebreak
$f=(16,108,184,92)$ and contains the double trefoil knot on three edges in its $1$-skeleton.
\end{thm}

If we remove from the $3$-sphere $S_{16,92}$ the facet $1\,9\,14\,15$,
then the resulting $3$-ball is LC, although it contains a double trefoil knot
as a three-edge subcomplex.

\begin{prop} The removal of the tetrahedron 
$1\,9\,14\,15$ from $S_{16,92}$ 
yields a locally constructible $3$-ball $B_{16, 91}$ with $16$ vertices 
and $91$ tetrahedra. 
\end{prop}

\begin{proof} Let $D$ be the $2$-ball given by the triangles $1\,9\,15$, $ 1\,14\,15$ 
and $9\,14\,15$. Clearly $D$ is a subcomplex of the boundary of $B_{16,91}$; 
it is in fact equal to $\partial B_{16,91}$ minus the triangle $1\,9\,14$. 
Our goal is to show that $B_{16,91}$ collapses onto $D$. The following is a certificate that this is true:

\vskip2mm
\noindent {\footnotesize First phase (pairs ``triangle'' $\rightarrow$ ``tetrahedron''):
\setlength{\arraycolsep}{0.5em}
\[\!\!\!\!\!
\begin{array}{rl@{\hspace{2mm}}l@{\hspace{2mm}}@{\hspace{4mm}}r@{\hspace{2mm}}l@{\hspace{2mm}}l@{\hspace{4mm}}r@{\hspace{2mm}}l@{\hspace{2mm}}l@{\hspace{4mm}}r@{\hspace{2mm}}l@{\hspace{2mm}}l}
 1\,9\,14   & \rightarrow & 1\,8\,9\,14,    &
 8\,9\,14   & \rightarrow & 2\,8\,9\,14,    &
 1\,8\,9    & \rightarrow & 1\,7\,8\,9,     &
 2\,8\,9    & \rightarrow & 2\,7\,8\,9,     \\

 1\,7\,8    & \rightarrow & 1\,3\,7\,8,     &
 1\,3\,7    & \rightarrow & 1\,3\,7\,11,    &
 3\,7\,8    & \rightarrow & 3\,5\,7\,8,     &
 1\,3\,8    & \rightarrow & 1\,3\,8\,11,    \\

 2\,7\,8    & \rightarrow & 2\,4\,7\,8,     &
 2\,8\,14   & \rightarrow & 2\,8\,12\,14,   &
 1\,7\,11   & \rightarrow & 1\,7\,9\,11,    &
 3\,5\,7    & \rightarrow & 3\,5\,6\,7,     \\

 5\,6\,7    & \rightarrow & 4\,5\,6\,7,     &
 3\,8\,11   & \rightarrow & 3\,5\,8\,11,    &
 2\,4\,8    & \rightarrow & 2\,4\,8\,16,    &
 3\,7\,11   & \rightarrow & 3\,7\,11\,14,   \\

 2\,7\,9    & \rightarrow & 2\,7\,9\,10,    &
 5\,8\,11   & \rightarrow & 5\,8\,11\,13,   &
 1\,8\,11   & \rightarrow & 1\,8\,11\,15,   &
 4\,8\,16   & \rightarrow & 4\,5\,8\,16,    \\

 1\,8\,15   & \rightarrow & 1\,8\,10\,15,   &
 2\,9\,14   & \rightarrow & 2\,6\,9\,14,    &
 8\,11\,15  & \rightarrow & 8\,11\,13\,15,  &
 5\,8\,16   & \rightarrow & 5\,8\,13\,16,   \\

 5\,13\,16  & \rightarrow & 1\,5\,13\,16,   &
 2\,6\,9    & \rightarrow & 2\,6\,9\,16,    &
 1\,9\,11   & \rightarrow & 1\,9\,11\,15,   &
 11\,13\,15 & \rightarrow & 9\,11\,13\,15,  \\

 1\,8\,10   & \rightarrow & 1\,8\,10\,14,   &
 6\,9\,16   & \rightarrow & 3\,6\,9\,16,    &
 2\,4\,7    & \rightarrow & 2\,4\,7\,15,    &
 2\,4\,15   & \rightarrow & 2\,3\,4\,15,    \\

 1\,10\,14  & \rightarrow & 1\,10\,13\,14,  &
 4\,5\,7    & \rightarrow & 4\,5\,7\,8,     &
 2\,9\,16   & \rightarrow & 2\,9\,10\,16,   &
 3\,5\,11   & \rightarrow & 3\,5\,11\,14,   \\

 1\,5\,13   & \rightarrow & 1\,5\,12\,13,   &
 8\,10\,14  & \rightarrow & 8\,10\,12\,14,  &
 4\,5\,16   & \rightarrow & 1\,4\,5\,16,    &
 1\,4\,5    & \rightarrow & 1\,4\,5\,6,     \\

 3\,5\,6    & \rightarrow & 3\,5\,6\,14,    &
 1\,13\,14  & \rightarrow & 1\,13\,14\,16,  &
 1\,5\,12   & \rightarrow & 1\,2\,5\,12,    &
 7\,11\,14  & \rightarrow & 7\,11\,12\,14,  \\

 9\,13\,15  & \rightarrow & 3\,9\,13\,15,   &
 3\,9\,13   & \rightarrow & 3\,9\,12\,13,   &
 1\,5\,6    & \rightarrow & 1\,2\,5\,6,     &
 3\,7\,14   & \rightarrow & 3\,7\,14\,16,   \\

 7\,14\,16  & \rightarrow & 7\,13\,14\,16,  &
 5\,6\,14   & \rightarrow & 2\,5\,6\,14,    &
 2\,3\,4    & \rightarrow & 2\,3\,4\,13,    &
 1\,4\,6    & \rightarrow & 1\,4\,6\,12,    \\

 6\,7\,13   & \rightarrow & 6\,7\,13\,16,   &
 8\,10\,15  & \rightarrow & 8\,10\,12\,15,  &
 4\,6\,7    & \rightarrow & 4\,6\,7\,15,    &
 1\,4\,16   & \rightarrow & 1\,4\,10\,16,   \\

 5\,11\,13  & \rightarrow & 5\,11\,12\,13,  &
 4\,6\,12   & \rightarrow & 4\,6\,12\,15,   &
 2\,10\,16  & \rightarrow & 2\,4\,10\,16,   &
 11\,12\,14 & \rightarrow & 5\,11\,12\,14,  \\

 2\,5\,12   & \rightarrow & 2\,5\,12\,14,   &
 4\,12\,15  & \rightarrow & 3\,4\,12\,15,   &
 3\,12\,13  & \rightarrow & 3\,4\,12\,13,   &
 3\,9\,16   & \rightarrow & 3\,9\,12\,16,   \\

 3\,12\,15  & \rightarrow & 3\,12\,15\,16,  &
 3\,6\,14   & \rightarrow & 3\,6\,9\,14,    &
 10\,12\,15 & \rightarrow & 10\,12\,15\,16, &
 6\,7\,16   & \rightarrow & 3\,6\,7\,16,    \\

 10\,12\,14 & \rightarrow & 7\,10\,12\,14,  &
 1\,10\,16  & \rightarrow & 1\,10\,15\,16,  &
 1\,6\,12   & \rightarrow & 1\,2\,6\,12,    &
 7\,10\,12  & \rightarrow & 7\,9\,10\,12,   \\

 7\,13\,15  & \rightarrow & 2\,7\,13\,15,   &
 2\,8\,16   & \rightarrow & 2\,6\,8\,16,    &
 9\,10\,16  & \rightarrow & 9\,10\,12\,16,  &
 2\,4\,13   & \rightarrow & 2\,4\,10\,13,   \\

 6\,13\,16  & \rightarrow & 6\,8\,13\,16,   &
 7\,9\,12   & \rightarrow & 7\,9\,11\,12,   &
 3\,13\,15  & \rightarrow & 2\,3\,13\,15,   &
 4\,12\,13  & \rightarrow & 1\,4\,12\,13,   \\

 4\,10\,13  & \rightarrow & 1\,4\,10\,13,   &
 6\,8\,15   & \rightarrow & 6\,8\,12\,15,   &
 2\,8\,12   & \rightarrow & 2\,6\,8\,12,    &
 3\,9\,15   & \rightarrow & 3\,9\,14\,15,   \\

 3\,14\,16  & \rightarrow & 3\,14\,15\,16,  &
 1\,14\,16  & \rightarrow & 1\,14\,15\,16.  &
\end{array}
\]

\noindent Second phase (pairs ``edge'' $\rightarrow$ ``triangle''):
\setlength{\arraycolsep}{0.2em}
\[\!\!\!\!\!\!\!
\begin{array}{r@{\hspace{2mm}}l@{\hspace{2mm}}l@{\hspace{4mm}}r@{\hspace{2mm}}l@{\hspace{2mm}}l@{\hspace{4mm}}r@{\hspace{2mm}}l@{\hspace{2mm}}l@{\hspace{4mm}}r@{\hspace{2mm}}l@{\hspace{2mm}}l@{\hspace{4mm}}r@{\hspace{2mm}}l@{\hspace{2mm}}l}
 8\,9   & \rightarrow & 7\,8\,9,    &
 2\,9   & \rightarrow & 2\,9\,10,   &
 1\,6   & \rightarrow & 1\,2\,6,    &
 5\,16  & \rightarrow & 1\,5\,16,   &
 3\,8   & \rightarrow & 3\,5\,8,    \\

 8\,15  & \rightarrow & 8\,12\,15,  &
 1\,3   & \rightarrow &  1\,3\,11,  &
 13\,15 & \rightarrow &  2\,13\,15, &
 1\,5   & \rightarrow & 1\,2\,5,    &
 1\,7   & \rightarrow & 1\,7\,9,    \\

 8\,10  & \rightarrow & 8\,10\,12,  &
 1\,8   & \rightarrow & 1\,8\,14,   &
 8\,11  & \rightarrow & 8\,11\,13,  &
 2\,8   & \rightarrow & 2\,6\,8,    &
 3\,5   & \rightarrow & 3\,5\,14,   \\

 11\,13 & \rightarrow & 11\,12\,13, &
 1\,2   & \rightarrow & 1\,2\,12,   &
 8\,14  & \rightarrow & 8\,12\,14,  &
 5\,7   & \rightarrow & 5\,7\,8,    &
 9\,16  & \rightarrow & 9\,12\,16,  \\

 6\,13  & \rightarrow & 6\,8\,13,   &
 3\,11  & \rightarrow & 3\,11\,14,  &
 8\,12  & \rightarrow & 6\,8\,12,   &
 9\,13  & \rightarrow & 9\,12\,13,  &
 7\,8   & \rightarrow & 4\,7\,8,    \\

 1\,11  & \rightarrow & 1\,11\,15,  &
 11\,14 & \rightarrow & 5\,11\,14,  &
 11\,15 & \rightarrow & 9\,11\,15,  &
 4\,8   & \rightarrow & 4\,5\,8,    &
 4\,5   & \rightarrow & 4\,5\,6,    \\

 5\,11  & \rightarrow & 5\,11\,12,  &
 5\,6   & \rightarrow & 2\,5\,6,    &
 4\,7   & \rightarrow & 4\,7\,15,   &
 5\,8   & \rightarrow & 5\,8\,13,   &
 2\,5   & \rightarrow & 2\,5\,14,   \\

 4\,6   & \rightarrow & 4\,6\,15,   &
 5\,13  & \rightarrow & 5\,12\,13,  &
 5\,12  & \rightarrow & 5\,12\,14,  &
 6\,8   & \rightarrow & 6\,8\,16,   &
 4\,15  & \rightarrow & 3\,4\,15,   \\

 12\,13 & \rightarrow & 1\,12\,13,  &
 8\,13  & \rightarrow & 8\,13\,16,  &
 1\,12  & \rightarrow & 1\,4\,12,   &
 4\,12  & \rightarrow & 3\,4\,12,   &
 3\,4   & \rightarrow & 3\,4\,13,   \\

 3\,13  & \rightarrow & 2\,3\,13,   &
 2\,3   & \rightarrow & 2\,3\,15,   &
 2\,15  & \rightarrow & 2\,7\,15,   &
 4\,13  & \rightarrow & 1\,4\,13,   &
 7\,15  & \rightarrow & 6\,7\,15,   \\

 6\,7   & \rightarrow & 3\,6\,7,    &
 2\,7   & \rightarrow & 2\,7\,13,   &
 6\,15  & \rightarrow & 6\,12\,15,  &
 6\,12  & \rightarrow & 2\,6\,12,   &
 3\,7   & \rightarrow & 3\,7\,16,   \\

 1\,4   & \rightarrow & 1\,4\,10,   &
 12\,15 & \rightarrow & 12\,15\,16, &
 7\,16  & \rightarrow &  7\,13\,16, &
 2\,13  & \rightarrow & 2\,10\,13,  &
 2\,12  & \rightarrow & 2\,12\,14,  \\

 2\,10  & \rightarrow & 2\,4\,10,   &
 2\,14  & \rightarrow & 2\,6\,14,   &
 4\,10  & \rightarrow & 4\,10\,16,  &
 2\,4   & \rightarrow & 2\,4\,16,   &
 12\,14 & \rightarrow & 7\,12\,14,  \\

 2\,6   & \rightarrow & 2\,6\,16,   &
 7\,13  & \rightarrow & 7\,13\,14,  &
 6\,16  & \rightarrow & 3\,6\,16,   &
 7\,14  & \rightarrow & 7\,10\,14,  &
 6\,14  & \rightarrow & 6\,9\,14,   \\

 6\,9   & \rightarrow & 3\,6\,9,    &
 7\,12  & \rightarrow & 7\,11\,12,  &
 7\,10  & \rightarrow & 7\,9\,10,   &
 7\,9   & \rightarrow & 7\,9\,11,   &
 9\,11  & \rightarrow & 9\,11\,12,  \\

 10\,14 & \rightarrow & 10\,13\,14, &
 9\,10  & \rightarrow & 9\,10\,12,  &
 9\,12  & \rightarrow & 3\,9\,12,   &
 10\,12 & \rightarrow & 10\,12\,16, &
 12\,16 & \rightarrow & 3\,12\,16,  \\

 13\,14 & \rightarrow & 13\,14\,16, &
 10\,13 & \rightarrow & 1\,10\,13,  &
 3\,9   & \rightarrow & 3\,9\,14,   &
 14\,16 & \rightarrow & 14\,15\,16, &
 1\,10  & \rightarrow & 1\,10\,15,  \\

 13\,16 & \rightarrow & 1\,13\,16,  &
 10\,16 & \rightarrow & 10\,15\,16, &
 3\,14  & \rightarrow & 3\,14\,15,  &
 3\,15  & \rightarrow & 3\,15\,16,  &
 1\,16  & \rightarrow & 1\,15\,16.
\end{array}
\]

\noindent Third phase (pairs ``vertex'' $\rightarrow$ ``edge''):
\setlength{\arraycolsep}{0.4em}
\[\!\!\!\!\!\!\!\!\!\!\!\!
\begin{array}{r@{\hspace{2mm}}l@{\hspace{2mm}}l@{\hspace{3mm}}r@{\hspace{2mm}}l@{\hspace{2mm}}l@{\hspace{3mm}}r@{\hspace{2mm}}l@{\hspace{2mm}}l@{\hspace{3mm}}r@{\hspace{2mm}}l@{\hspace{2mm}}l@{\hspace{3mm}}r@{\hspace{2mm}}l@{\hspace{2mm}}l@{\hspace{3mm}}r@{\hspace{2mm}}l@{\hspace{2mm}}l@{\hspace{3mm}}r@{\hspace{2mm}}l@{\hspace{2mm}}l}
 13 & \rightarrow & 1\,13,  &
 5  & \rightarrow & 5\,14,  &
 6  & \rightarrow & 3\,6,   &
 10 & \rightarrow & 10\,15, &
 7  & \rightarrow & 7\,11,  &
 11 & \rightarrow & 11\,12, &
 12 & \rightarrow & 3\,12,  \\

 2  & \rightarrow & 2\,16,  &
 3  & \rightarrow & 3\,16,  &
 4  & \rightarrow & 4\,16,  &
 8  & \rightarrow & 8\,16,  &
 16 & \rightarrow & 15\,16. \\
\end{array}
\]
}%
\end{proof}

If we remove from the $3$-sphere $S_{16,92}$ the entire star of the vertex $1$ (one of the three knot vertices),
we obtain a $3$-ball $B_{15,66}$ with many interesting properties.
In the following we will show that 
\begin{compactenum}[(1)]
\item $B_{15,66}$ contains a knotted spanning edge $2\,3$, where the knot is the double trefoil;
\item $B_{15,66}$  is not embeddable in $\mathbb{R}^3$;
\item $B_{15,66}$  is not collapsible, but it admits a discrete Morse function with one critical vertex, one critical edge and one critical triangle.
\end{compactenum}

\begin{prop}
$B_{15,66}$ is not rectilinearly-embeddable in $\mathbb{R}^3$. 
\end{prop}

\begin{proof}
The boundary of $B_{15,66}$ consists of the following $26$ triangles:
{\small
\[
\begin{array}{l@{\hspace{3.5mm}}l@{\hspace{2mm}}l@{\hspace{2mm}}l@{\hspace{2mm}}l@{\hspace{2mm}}l@{\hspace{2mm}}l@{\hspace{2mm}}l@{\hspace{2mm}}l}
  2\,5\,6,   & 2\,5\,12,  & 2\,6\,12,   & 3\,7\,8,   & 3\,7\,11,   & 3\,8\,11,   & 4\,5\,6,    & 4\,5\,16,    & 4\,6\,12,  \\

  4\,10\,13, & 4\,10\,16, &  4\,12\,13, & 5\,12\,13, & 5\,13\,16,  & 7\,8\,9,    & 7\,9\,11,   & 8\,9\,14,    & 8\,10\,14, \\

  8\,10\,15, & 8\,11\,15, & 9\,11\,15,  & 9\,14\,15, & 10\,13\,14, & 10\,15\,16, & 13\,14\,16, & 14\,15\,16. &
\end{array}
\]
}%
In particular, the five edges $2\,5$, $5\,13$, $10\,13$, $8\,10$ and $3\,8$ 
form a boundary path from the vertex $2$ to the vertex $3$. 
Together with the interior edge $2\,3$, this path closes up to a hexagonal double trefoil knot. 
By Theorem~\ref{thm:embed}, $B_{15,66}$ cannot be rectilinearly embedded in $\mathbb{R}^3$.
\end{proof}

\begin{thm}
$B_{15,66}$ admits a discrete Morse function with three critical faces, all of them  
belonging to the boundary $\partial B_{15,66}$.
\end{thm}

\begin{proof}
We will show that there is a $2$-dimensional subcomplex $C$ of $B_{15,66}$ such that:
\begin{compactitem}
\item $B_{15,66}$ collapses onto $C$ and
\item $C$ minus the triangle $2\,5\,8$ collapses onto a pentagon.
\end{compactitem}
Here is the right collapsing sequence:

\vskip2mm
\noindent {\footnotesize First phase (pairs ``triangle'' $\rightarrow$ ``tetrahedron''):
\setlength{\arraycolsep}{0.5em}
\[\!\!\!\!\!
\begin{array}{r@{\hspace{2mm}}l@{\hspace{2mm}}l@{\hspace{5mm}}r@{\hspace{2mm}}l@{\hspace{2mm}}l@{\hspace{5mm}}r@{\hspace{2mm}}l@{\hspace{2mm}}l@{\hspace{5mm}}r@{\hspace{2mm}}l@{\hspace{2mm}}l}
 4\,10\,16  & \rightarrow & 2\,4\,10\,16, &
 4\,10\,13  & \rightarrow & 2\,4\,10\,13, &
 9\,14\,15  & \rightarrow & 3\,9\,14\,15, &
 10\,15\,16 & \rightarrow & 10\,12\,15\,16, \\

 8\,11\,15  & \rightarrow & 8\,11\,13\,15, &
 3\,8\,11   & \rightarrow & 3\,5\,8\,11, &
 8\,13\,15  & \rightarrow & 6\,8\,13\,15, &
 13\,14\,16 & \rightarrow & 7\,13\,14\,16, \\

 4\,5\,16   & \rightarrow & 4\,5\,8\,16, &
 6\,8\,15   & \rightarrow & 6\,8\,12\,15, &
 4\,5\,6    & \rightarrow & 4\,5\,6\,7, &
 8\,10\,15  & \rightarrow & 8\,10\,12\,15, \\

 8\,9\,14   & \rightarrow & 2\,8\,9\,14, &
 2\,4\,13   & \rightarrow & 2\,3\,4\,13, &
 14\,15\,16 & \rightarrow & 3\,14\,15\,16, &
 2\,5\,12   & \rightarrow & 2\,5\,12\,14, \\

 4\,8\,16   & \rightarrow & 2\,4\,8\,16, &
 2\,8\,14   & \rightarrow & 2\,8\,12\,14, &
 2\,4\,8    & \rightarrow & 2\,4\,7\,8, &
 8\,10\,12  & \rightarrow & 8\,10\,12\,14, \\

 2\,3\,13   & \rightarrow & 2\,3\,13\,15, &
 3\,7\,11   & \rightarrow & 3\,7\,11\,14, &
 4\,6\,12   & \rightarrow & 4\,6\,12\,15, &
 2\,6\,12   & \rightarrow & 2\,6\,8\,12, \\

 9\,11\,15  & \rightarrow & 9\,11\,13\,15, &
 2\,8\,16   & \rightarrow & 2\,6\,8\,16, &
 4\,12\,15  & \rightarrow & 3\,4\,12\,15, &
 2\,8\,9    & \rightarrow & 2\,7\,8\,9, \\

 3\,14\,16  & \rightarrow &  3\,7\,14\,16, &
 4\,5\,8    & \rightarrow & 4\,5\,7\,8, &
 5\,6\,7    & \rightarrow & 3\,5\,6\,7, &
 3\,5\,6    & \rightarrow & 3\,5\,6\,14, \\

 6\,8\,13   & \rightarrow & 6\,8\,13\,16, &
 3\,13\,15  & \rightarrow & 3\,9\,13\,15, &
 3\,4\,13   & \rightarrow & 3\,4\,12\,13, &
 5\,8\,16   & \rightarrow & 5\,8\,13\,16, \\

 2\,4\,7    & \rightarrow & 2\,4\,7\,15, &
 5\,12\,14  & \rightarrow & 5\,11\,12\,14, &
 3\,5\,7    & \rightarrow & 3\,5\,7\,8, &
 2\,6\,16   & \rightarrow & 2\,6\,9\,16, \\ 

 12\,15\,16 & \rightarrow & 3\,12\,15\,16, &
 2\,4\,15   & \rightarrow & 2\,3\,4\,15, &
 6\,13\,16  & \rightarrow & 6\,7\,13\,16, &
 2\,9\,14   & \rightarrow &  2\,6\,9\,14, \\

 2\,9\,16   & \rightarrow & 2\,9\,10\,16, &
 2\,6\,14   & \rightarrow & 2\,5\,6\,14, &
 5\,12\,13  & \rightarrow & 5\,11\,12\,13, &
 3\,12\,16  & \rightarrow & 3\,9\,12\,16, \\

 7\,13\,14  & \rightarrow & 7\,10\,13\,14, &
 3\,11\,14  & \rightarrow & 3\,5\,11\,14, &
 7\,11\,14  & \rightarrow & 7\,11\,12\,14, &
 5\,11\,13  & \rightarrow & 5\,8\,11\,13, \\

 9\,12\,16  & \rightarrow & 9\,10\,12\,16, &
 7\,10\,13  & \rightarrow & 2\,7\,10\,13, &
 9\,10\,12  & \rightarrow & 7\,9\,10\,12, &
 7\,9\,10   & \rightarrow & 2\,7\,9\,10, \\

 3\,6\,14   & \rightarrow & 3\,6\,9\,14, &
 3\,9\,16   & \rightarrow & 3\,6\,9\,16, &
 3\,6\,7    & \rightarrow & 3\,6\,7\,16, &
 4\,7\,15   & \rightarrow & 4\,6\,7\,15, \\

 3\,12\,13  & \rightarrow & 3\,9\,12\,13, &
 7\,10\,12  & \rightarrow & 7\,10\,12\,14, &
 2\,13\,15  & \rightarrow & 2\,7\,13\,15, &
 7\,9\,12   & \rightarrow & 7\,9\,11\,12, \\

 9\,11\,12  & \rightarrow & 9\,11\,12\,13, &
 6\,7\,15   & \rightarrow & 6\,7\,13\,15. &
\end{array}
\]

\noindent Second phase (pairs ``edge'' $\rightarrow$ ``triangle''):
\setlength{\arraycolsep}{0.2em}
\[\!\!\!\!\!\!\!
\begin{array}{r@{\hspace{2mm}}l@{\hspace{2mm}}l@{\hspace{4mm}}r@{\hspace{2mm}}l@{\hspace{2mm}}l@{\hspace{4mm}}r@{\hspace{2mm}}l@{\hspace{2mm}}l@{\hspace{4mm}}r@{\hspace{2mm}}l@{\hspace{2mm}}l@{\hspace{4mm}}r@{\hspace{2mm}}l@{\hspace{2mm}}l}
 8\,9   & \rightarrow &  7\,8\,9,   &
 14\,16 & \rightarrow & 7\,14\,16,  &
 4\,5   & \rightarrow & 4\,5\,7,    &
 3\,11  & \rightarrow & 3\,5\,11,   &
 10\,15 & \rightarrow & 10\,12\,15, \\

 14\,15 & \rightarrow & 3\,14\,15,  &
 13\,14 & \rightarrow & 10\,13\,14, &
 5\,7   & \rightarrow & 5\,7\,8,    &
 8\,10  & \rightarrow & 8\,10\,14,  &
 8\,14  & \rightarrow & 8\,12\,14,  \\

 4\,10  & \rightarrow & 2\,4\,10,   &
 4\,13  & \rightarrow & 4\,12\,13,  &
 15\,16 & \rightarrow & 3\,15\,16,  &
 5\,16  & \rightarrow & 5\,13\,16,  &
 4\,16  & \rightarrow & 2\,4\,16,   \\

 12\,16 & \rightarrow & 10\,12\,16, &
 4\,8   & \rightarrow & 4\,7\,8,    &
 10\,13 & \rightarrow & 2\,10\,13,  &
 10\,12 & \rightarrow & 10\,12\,14, &
 2\,4   & \rightarrow & 2\,3\,4,    \\

 4\,12  & \rightarrow & 3\,4\,12,   &
 2\,16  & \rightarrow & 2\,10\,16,  &
 2\,3   & \rightarrow & 2\,3\,15,   &
 5\,13  & \rightarrow & 5\,8\,13,   &
 2\,15  & \rightarrow & 2\,7\,15,   \\

 11\,15 & \rightarrow & 11\,13\,15, &
 10\,16 & \rightarrow & 9\,10\,16,  &
 3\,13  & \rightarrow & 3\,9\,13,   &
 9\,16  & \rightarrow & 6\,9\,16,   &
 9\,10  & \rightarrow & 2\,9\,10,   \\

 10\,14 & \rightarrow & 7\,10\,14,  &
 2\,13  & \rightarrow & 2\,7\,13,   &
 7\,15  & \rightarrow & 7\,13\,15,  &
 7\,10  & \rightarrow & 2\,7\,10,   &
 4\,7   & \rightarrow & 4\,6\,7,    \\

 8\,15  & \rightarrow & 8\,12\,15,  &
 4\,6   & \rightarrow & 4\,6\,15,   &
 4\,15  & \rightarrow & 3\,4\,15,   &
 5\,12  & \rightarrow & 5\,11\,12.  &
\end{array}
\]
}%
Let $C$ be the obtained $2$-complex. Note that $C$ contains the triangle $2\,5\,8$, 
which belongs to $\partial B_{15,66}$ and has not been collapsed yet. 
Let $D$ be the complex obtained from $C$ after removing the (interior of the) 
triangle $2\,5\,8$.  Here is a proof:

\vskip2mm
\noindent {\footnotesize First phase (pairs ``edge'' $\rightarrow$ ``triangle''):
\setlength{\arraycolsep}{0.2em}
\[\!\!\!\!\!\!\!
\begin{array}{r@{\hspace{2mm}}l@{\hspace{2mm}}l@{\hspace{4mm}}r@{\hspace{2mm}}l@{\hspace{2mm}}l@{\hspace{4mm}}r@{\hspace{2mm}}l@{\hspace{2mm}}l@{\hspace{4mm}}r@{\hspace{2mm}}l@{\hspace{2mm}}l@{\hspace{4mm}}r@{\hspace{2mm}}l@{\hspace{2mm}}l}
 2\,5   & \rightarrow & 2\,5\,14,   &
 2\,14  & \rightarrow & 2\,12\,14,  &
 5\,6   & \rightarrow & 5\,6\,14,   &
 6\,14  & \rightarrow & 6\,9\,14,   &
 9\,14  & \rightarrow & 3\,9\,14,   \\

 2\,12  & \rightarrow & 2\,8\,12,   &
 8\,12  & \rightarrow & 6\,8\,12,   &
 6\,12  & \rightarrow & 6\,12\,15,  &
 6\,15  & \rightarrow & 6\,13\,15,  &
 12\,15 & \rightarrow & 3\,12\,15,  \\

 3\,15  & \rightarrow & 3\,9\,15,   &
 13\,15 & \rightarrow & 9\,13\,15,  &
 6\,13  & \rightarrow & 6\,7\,13,   &
 3\,12  & \rightarrow & 3\,9\,12,   &
 9\,12  & \rightarrow & 9\,12\,13,  \\

 3\,9   & \rightarrow & 3\,6\,9,    &
 7\,13  & \rightarrow & 7\,13\,16,  &
 3\,6   & \rightarrow & 3\,6\,16,   &
 3\,16  & \rightarrow & 3\,7\,16,   &
 12\,13 & \rightarrow & 11\,12\,13, \\

 6\,7   & \rightarrow & 6\,7\,16,   &
 6\,16  & \rightarrow & 6\,8\,16,   &
 8\,16  & \rightarrow & 8\,13\,16,  &
 6\,8   & \rightarrow & 2\,6\,8,    &
 9\,13  & \rightarrow & 9\,11\,13,  \\

 2\,6   & \rightarrow & 2\,6\,9,    &
 8\,13  & \rightarrow & 8\,11\,13,  &
 2\,8   & \rightarrow & 2\,7\,8,    &
 7\,8   & \rightarrow & 3\,7\,8,    &
 3\,8   & \rightarrow & 3\,5\,8,    \\

 3\,7   & \rightarrow & 3\,7\,14,   &
 3\,5   & \rightarrow & 3\,5\,14,   &
 8\,11  & \rightarrow & 5\,8\,11,   &
 5\,14  & \rightarrow & 5\,11\,14,  &
 2\,9   & \rightarrow & 2\,7\,9,    \\

 9\,11  & \rightarrow & 7\,9\,11,   &
 7\,11  & \rightarrow & 7\,11\,12,  &
 11\,12 & \rightarrow & 11\,12\,14, &
 7\,12  & \rightarrow & 7\,12\,14.  &
\end{array}
\]

\noindent Final phase (pairs ``vertex'' $\rightarrow$ ``edge''):
\setlength{\arraycolsep}{0.4em}
\[\!\!\!\!\!\!\!\!\!\!\!\!
\begin{array}{r@{\hspace{2mm}}l@{\hspace{2mm}}l@{\hspace{3mm}}r@{\hspace{2mm}}l@{\hspace{2mm}}l@{\hspace{3mm}}r@{\hspace{2mm}}l@{\hspace{2mm}}l@{\hspace{3mm}}r@{\hspace{2mm}}l@{\hspace{2mm}}l@{\hspace{3mm}}r@{\hspace{2mm}}l@{\hspace{2mm}}l@{\hspace{3mm}}r@{\hspace{2mm}}l@{\hspace{2mm}}l@{\hspace{3mm}}r@{\hspace{2mm}}l@{\hspace{2mm}}l@{\hspace{3mm}}r@{\hspace{2mm}}l@{\hspace{2mm}}l}
 2  & \rightarrow & 2\,7,   &
 15 & \rightarrow & 9\,15,  &
 3  & \rightarrow & 3\,14,  &
 12 & \rightarrow & 12\,14, &
 6  & \rightarrow & 6\,9,   &
 8  & \rightarrow & 5\,8,   &
 5  & \rightarrow & 5\,11,  &
 9  & \rightarrow & 7\,9.
\end{array}
\]
}%

At this point we are left with the pentagon $P$ given 
by the five edges $7\,14$, $7\,16$, $11\,13$, $11\,14$, and $13\,16$. 
The latter edge, $13\,16$, belongs to the boundary of $B_{15,66}$. Clearly, 
$P$ minus this edge yields a collapsible $1$-ball. Thus, $B_{15,66}$ 
admits a discrete Morse function whose critical faces are the vertex $13$, 
the edge $13\,66$ and the triangle $2\,5\,8$. This discrete Morse function 
is the best possible, since $B_{15,66}$ cannot be collapsible (because of
its knotted spanning edge $2\,3$).
\end{proof}

  \begin{figure}[t]
    \begin{center}
      \small
      \psfrag{1}{1}
      \psfrag{2}{2}
      \psfrag{3}{3}
      \psfrag{4}{4}
      \psfrag{5}{5}
      \psfrag{6}{6}
      \psfrag{7}{7}
      \psfrag{8}{8}
      \psfrag{9}{9}
      \psfrag{10}{10}
      \psfrag{11}{11}
      \psfrag{12}{12}
      \psfrag{13}{13}
      \psfrag{14}{14}
      \psfrag{15}{15}
      \psfrag{16}{16}
      \psfrag{17}{17}
      \psfrag{18}{18}
      \psfrag{19}{19}
      \psfrag{20}{20}
      \psfrag{21}{21}
      \psfrag{22}{22}
      \psfrag{23}{23}
      \psfrag{24}{24}
      \psfrag{25}{25}
      \psfrag{26}{26}
      \psfrag{27}{27}
      \psfrag{28}{28}
      \psfrag{29}{29}
      \psfrag{30}{30}
      \psfrag{31}{31}
      \psfrag{32}{32}
      \psfrag{33}{33}
      \psfrag{34}{34}
      \psfrag{35}{35}
      \psfrag{36}{36}
      \psfrag{37}{37}
      \psfrag{38}{38}
      \psfrag{39}{39}
      \psfrag{40}{40}
      \psfrag{41}{41}
      \psfrag{42}{42}
      \psfrag{43}{43}
      \includegraphics[width=12cm]{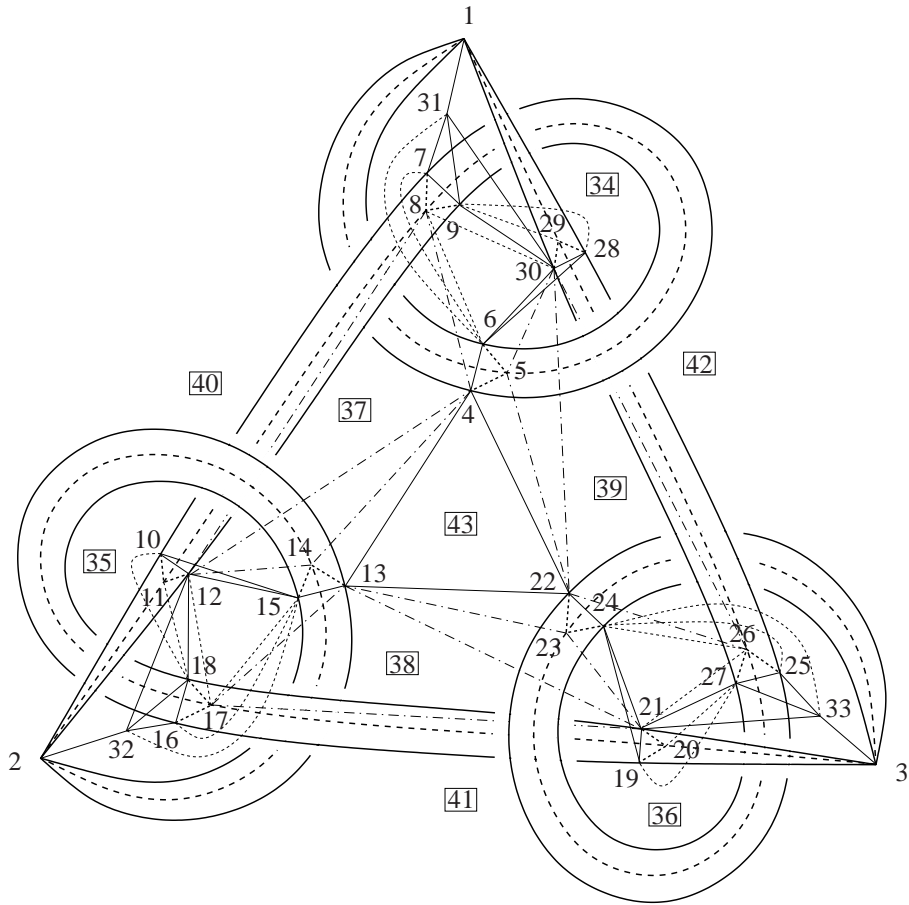}
    \end{center}
    \caption{The triple trefoil in the sphere  $S_{44,284}$.
      \label{fig:triple_trefoil}}
  \end{figure}

  \begin{figure}[htp]
    \begin{center}
      \small
      \psfrag{1}{1}
      \psfrag{2}{2}
      \psfrag{3}{3}
      \psfrag{4}{4}
      \psfrag{5}{5}
      \psfrag{6}{6}
      \psfrag{7}{7}
      \psfrag{8}{8}
      \psfrag{9}{9}
      \psfrag{10}{10}
      \psfrag{11}{11}
      \psfrag{12}{12}
      \psfrag{13}{13}
      \psfrag{14}{14}
      \psfrag{15}{15}
      \psfrag{16}{16}
      \psfrag{17}{17}
      \psfrag{18}{18}
      \psfrag{19}{19}
      \psfrag{20}{20}
      \psfrag{21}{21}
      \psfrag{22}{22}
      \psfrag{23}{23}
      \psfrag{24}{24}
      \psfrag{25}{25}
      \psfrag{26}{26}
      \psfrag{27}{27}
      \psfrag{28}{28}
      \psfrag{29}{29}
      \psfrag{30}{30}      
      \includegraphics[width=11cm]{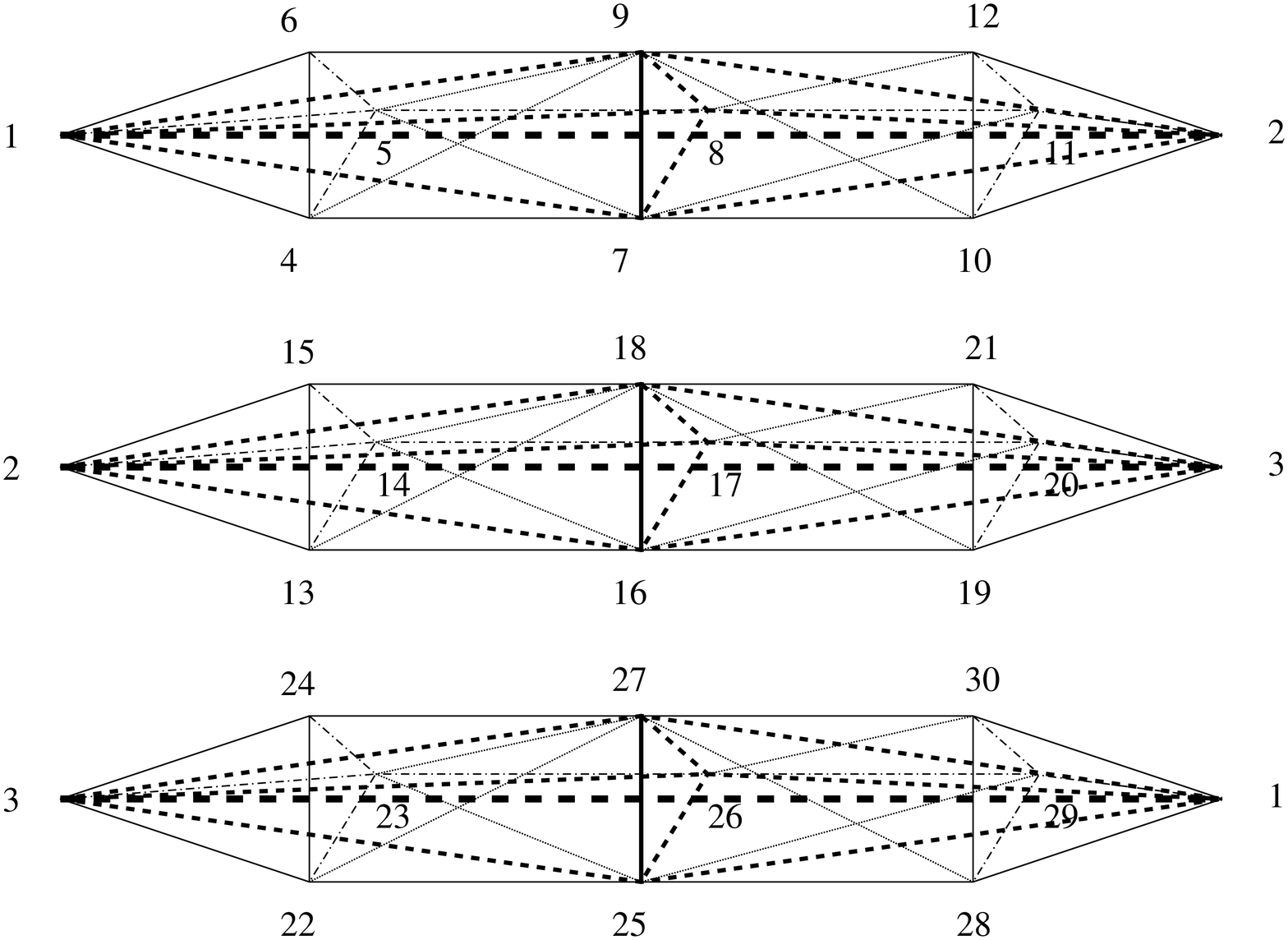}
    \end{center}
    \caption{The spindles of $S_{44,284}$.
      \label{fig:spindles_triple_trefoil}}
  \end{figure}

\section{The triple trefoil} \label{sec:3}

In this section, we are constructing a triangulation $S_{44,284}$ of the $3$-sphere $S^3$
that contains a triple trefoil knot with three edges in its $1$-skeleton. We then use bistellar flips
to obtain a reduced triangulation $S_{18,125}$.

As before for the double trefoil, we place a triple trefoil knot on the three edges $1\,2$, $2\,3$, $1\,3$ in ${\mathbb R}^3$,
as depicted in Figure~\ref{fig:triple_trefoil}.  Each of the three knot edges is protected by a spindle;
see Figure~\ref{fig:spindles_triple_trefoil} for the spindles and Table~\ref{tbl:spindles_triple_trefoil}
for the list of tetrahedra of the spindles.

To close the holes of the knot we glue in the membrane triangles of Table~\ref{tbl:triple_trefoil_membrane_triangles}
and then add the local cones with respect to the vertices $34$, $35$, \dots, $43$ from Table~\ref{tbl:triple_trefoil_local_cones}
to obtain a $3$-ball $B_{43,214}$.

Finally, we add to $B_{43,214}$ the cone over its boundary with respect to the vertex $44$ (as given in Table~\ref{tbl:triple_trefoil_cone})
to obtain the $3$-sphere $S_{44,284}$.

\begin{prop}
The $3$-sphere $S_{44,284}$ consists of $284$ tetrahedra and $44$ vertices.
It has face vector $f=(44,328,568,284)$ and contains the triple trefoil knot on three edges in its $1$-skeleton.
\end{prop}

Again, the $3$-sphere $S_{44,284}$ is not minimal with the property of containing the triple trefoil knot
in its $1$-skeleton. The smallest triangulation we found via bistellar flips is $S_{18,125}$; 
see Table~\ref{tbl:triple_trefoil_sphere} for the list of facets of $S_{18,125}$.

\begin{table}[htp]
\small\centering
\defaultaddspace=0.6em
\caption{Part A of the sphere $S_{44,284}$: The three spindles.}\label{tbl:spindles_triple_trefoil}
\begin{center}\footnotesize
\begin{tabular*}{\linewidth}{@{\extracolsep{\fill}}lll@{\hspace{9mm}}lll@{\hspace{9mm}}lll@{}}
\toprule
 \addlinespace
$1\,2\,7\,8$  &  $1\,4\,6\,9$  &  $2\,7\,9\,10$     &    $2\,3\,16\,17$  &  $2\,13\,15\,18$  &  $3\,16\,18\,19$    &    $1\,3\,25\,26$  &  $3\,22\,24\,27$  &  $1\,25\,27\,28$  \\
$1\,2\,7\,9$  &  $1\,4\,7\,9$  &  $2\,9\,10\,12$    &    $2\,3\,16\,18$  &  $2\,13\,16\,18$  &  $3\,18\,19\,21$    &    $1\,3\,25\,27$  &  $3\,22\,25\,27$  &  $1\,27\,28\,30$  \\
$1\,2\,8\,9$  &  $1\,4\,5\,7$  &  $2\,7\,8\,11$     &    $2\,3\,17\,18$  &  $2\,13\,14\,16$  &  $3\,16\,17\,20$    &    $1\,3\,26\,27$  &  $3\,22\,23\,25$  &  $1\,25\,26\,29$  \\
              &  $1\,5\,7\,8$  &  $2\,7\,10\,11$    &                    &  $2\,14\,16\,17$  &  $3\,16\,19\,20$    &                    &  $3\,23\,25\,26$  &  $1\,25\,28\,29$  \\
              &  $1\,5\,6\,9$  &  $2\,8\,9\,12$     &                    &  $2\,14\,15\,18$  &  $3\,17\,18\,21$    &                    &  $3\,23\,24\,27$  &  $1\,26\,27\,30$  \\
              &  $1\,5\,8\,9$  &  $2\,8\,11\,12$    &                    &  $2\,14\,17\,18$  &  $3\,17\,20\,21$    &                    &  $3\,23\,26\,27$  &  $1\,26\,29\,30$  \\
 \addlinespace
\bottomrule
\end{tabular*}
\end{center}
\end{table}

\begin{table}[htp]
\small\centering
\defaultaddspace=0.6em
\caption{The triangles of the membranes in the sphere $S_{44,284}$.}\label{tbl:triple_trefoil_membrane_triangles}
\begin{center}\footnotesize
\begin{tabular*}{\linewidth}{@{\extracolsep{\fill}}lll@{\hspace{13mm}}lll@{\hspace{13mm}}lll@{}}
\toprule
 \addlinespace
                   &                &                 &                  &   $4\,13\,22$ \\[2mm]
      $8\,9\,30$   &  $6\,28\,30$   &  $1\,30\,31$    &    $12\,17\,18$  &  $10\,12\,15$  &  $2\,12\,32$     &    $21\,26\,27$  &  $19\,21\,24$  &  $3\,21\,33$   \\
      $6\,8\,30$   &  $6\,9\,28$    &  $9\,30\,31$    &    $12\,15\,17$  &  $10\,15\,18$  &  $12\,18\,32$    &    $21\,24\,26$  &  $19\,24\,27$  &  $21\,27\,33$  \\
      $4\,6\,8$    &  $9\,28\,29$   &  $7\,9\,31$     &    $13\,15\,17$  &  $10\,11\,18$  &  $16\,18\,32$    &    $22\,24\,26$  &  $19\,20\,27$  &  $25\,27\,33$  \\
      $4\,8\,12$   &  $9\,29\,30$   &  $1\,6\,31$     &    $13\,17\,21$  &  $11\,12\,18$  &  $2\,15\,32$     &    $22\,26\,30$  &  $20\,21\,27$  &  $3\,24\,33$   \\
      $4\,12\,14$  &                &  $6\,7\,31$     &    $13\,21\,23$  &                &  $15\,16\,32$    &    $5\,22\,30$   &                &  $24\,25\,33$  \\
      $4\,13\,14$  &                &  $6\,7\,8$      &    $13\,22\,23$  &                &  $15\,16\,17$    &    $4\,5\,22$    &                &  $24\,25\,26$  \\
      $12\,14\,15$ &                &                 &    $21\,23\,24$  &                &                  &    $5\,6\,30$    &                & \\
 \addlinespace
\bottomrule
\end{tabular*}
\end{center}
\end{table}

\begin{table}[htp]
\small\centering
\defaultaddspace=0.6em
\caption{Part B of the sphere $S_{44,284}$: Tetrahedra to thicken Part A to a ball $B_{43,214}$.}\label{tbl:triple_trefoil_local_cones}
\begin{center}\footnotesize
\begin{tabular*}{\linewidth}{@{\extracolsep{\fill}}lll@{\hspace{14mm}}lll@{}}
\toprule
 \addlinespace
   $9\,29\,30\,31$  & $11\,12\,18\,32$ &  $20\,21\,27\,33$   &    $1\,31\,37\,40$  & $2\,32\,38\,41$  & $3\,33\,39\,42$  \\
   $1\,29\,30\,31$  & $2\,11\,12\,32$  &  $3\,20\,21\,33$    &    $7\,31\,37\,40$  & $16\,32\,38\,41$ & $25\,33\,39\,42$ \\
   $7\,9\,31\,34$   & $16\,18\,32\,35$ &  $25\,27\,33\,36$   &    $7\,10\,37\,40$  & $16\,19\,38\,41$ & $25\,28\,39\,42$ \\
   $9\,29\,31\,34$  & $11\,18\,32\,35$ &  $20\,27\,33\,36$   &    $10\,15\,37\,40$ & $19\,24\,38\,41$ & $6\,28\,39\,42$  \\
   $1\,29\,31\,34$  & $2\,11\,32\,35$  &  $3\,20\,33\,36$    &    $14\,15\,37\,40$ & $23\,24\,38\,41$ & $5\,6\,39\,42$   \\
   $1\,28\,29\,34$  & $2\,10\,11\,35$  &  $3\,19\,20\,36$    &    $14\,17\,37\,40$ & $23\,26\,38\,41$ & $5\,8\,39\,42$   \\
   $4\,6\,9\,34$    & $13\,15\,18\,35$ &  $22\,24\,27\,36$   &    $14\,15\,18\,40$ & $23\,24\,27\,41$ & $5\,6\,9\,42$    \\
   $4\,7\,9\,34$    & $13\,16\,18\,35$ &  $22\,25\,27\,36$   &    $14\,17\,18\,40$ & $23\,26\,27\,41$ & $5\,8\,9\,42$    \\
   $9\,28\,29\,34$  & $10\,11\,18\,35$ &  $19\,20\,27\,36$   &    $10\,15\,18\,40$ & $19\,24\,27\,41$ & $6\,9\,28\,42$   \\
   $6\,9\,28\,34$   & $10\,15\,18\,35$ &  $19\,24\,27\,36$   &    $10\,11\,18\,40$ & $19\,20\,27\,41$ & $9\,28\,29\,42$  \\
   $6\,28\,30\,34$  & $10\,12\,15\,35$ &  $19\,21\,24\,36$   &    $11\,12\,18\,40$ & $20\,21\,27\,41$ & $9\,29\,30\,42$  \\
   $1\,28\,30\,34$  & $2\,10\,12\,35$  &  $3\,19\,21\,36$    &    $8\,11\,12\,40$  & $17\,20\,21\,41$ & $26\,29\,30\,42$ \\
                    &                  &                     &    $7\,8\,11\,40$   & $16\,17\,20\,41$ & $25\,26\,29\,42$ \\
   $4\,7\,34\,39$   & $13\,16\,35\,37$ &  $22\,25\,36\,38$   &    $7\,10\,11\,40$  & $16\,19\,20\,41$ & $25\,28\,29\,42$ \\
   $4\,34\,37\,39$  & $13\,35\,37\,38$ &  $22\,36\,38\,39$   &    $6\,7\,8\,40$    & $15\,16\,17\,41$ & $24\,25\,26\,42$ \\
                    &                  &                     &    $6\,7\,31\,40$   & $15\,16\,32\,41$ & $24\,25\,33\,42$ \\
   $10\,12\,14\,15$ & $19\,21\,23\,24$ & $5\,6\,28\,30$      &    $1\,6\,31\,40$   & $2\,15\,32\,41$  & $3\,24\,33\,42$  \\
   $7\,9\,10\,37$   & $16\,18\,19\,38$ & $25\,27\,28\,39$    &    $1\,4\,6\,40$    & $2\,13\,15\,41$  & $3\,22\,24\,42$  \\
   $9\,10\,12\,37$  & $18\,19\,21\,38$ & $27\,28\,30\,39$    &    $4\,6\,8\,40$    & $13\,15\,17\,41$ & $22\,24\,26\,42$  \\
   $10\,12\,14\,37$ & $19\,21\,23\,38$ & $5\,28\,30\,39$     &    $4\,8\,12\,40$   & $13\,17\,21\,41$ & $22\,26\,30\,42$ \\
   $10\,14\,15\,37$ & $19\,23\,24\,38$ & $5\,6\,28\,39$      &                     &                  & \\
   $13\,14\,16\,37$ & $22\,23\,25\,38$ & $4\,5\,7\,39$       &                     & $4\,13\,22\,43$  & \\
   $14\,16\,17\,37$ & $23\,25\,26\,38$ & $5\,7\,8\,39$       &                     & $4\,13\,37\,43$  & \\
   $4\,13\,14\,37$  & $13\,22\,23\,38$ & $4\,5\,22\,39$      &                     & $13\,37\,38\,43$ & \\
   $4\,12\,14\,37$  & $13\,21\,23\,38$ & $5\,22\,30\,39$     &                     & $13\,22\,38\,43$ & \\
   $4\,8\,12\,37$   & $13\,17\,21\,38$ & $22\,26\,30\,39$    &                     & $22\,38\,39\,43$ & \\
   $8\,9\,12\,37$   & $17\,18\,21\,38$ & $26\,27\,30\,39$    &                     & $4\,22\,39\,43$  & \\
   $4\,6\,8\,37$    & $13\,15\,17\,38$ & $22\,24\,26\,39$    &                     & $4\,37\,39\,43$  & \\
   $6\,8\,30\,37$   & $12\,15\,17\,38$ & $21\,24\,26\,39$    &                     &                  & \\
   $8\,9\,30\,37$   & $12\,17\,18\,38$ & $21\,26\,27\,39$    &                     &                  & \\
   $9\,30\,31\,37$  & $12\,18\,32\,38$ & $21\,27\,33\,39$    &                     &                  & \\
   $1\,30\,31\,37$  & $2\,12\,32\,38$  & $3\,21\,33\,39$     &                     &                  & \\
   $7\,9\,31\,37$   & $16\,18\,32\,38$ & $25\,27\,33\,39$    &                     &                  & \\
   $1\,30\,34\,37$  & $2\,12\,35\,38$  & $3\,21\,36\,39$     &                     &                  & \\
   $6\,30\,34\,37$  & $12\,15\,35\,38$ & $21\,24\,36\,39$    &                     &                  & \\
   $4\,6\,34\,37$   & $13\,15\,35\,38$ & $22\,24\,36\,39$    &                     &                  & \\
 \addlinespace
\bottomrule
\end{tabular*}
\end{center}
\end{table}
\begin{table}[htp]
\small\centering
\defaultaddspace=0.6em
\caption{Part C of the sphere $S_{44,284}$: Cone over the boundary of the ball $B_{43,214}$.}\label{tbl:triple_trefoil_cone}
\begin{center}\footnotesize
\begin{tabular*}{\linewidth}{@{\extracolsep{\fill}}lllllll@{}}
\toprule
 \addlinespace
  $1\,4\,5\,44$    & $1\,4\,40\,44$   & $1\,5\,6\,44$    & $1\,6\,31\,44$   & $1\,31\,34\,44$  & $1\,34\,37\,44$  & $1\,37\,40\,44$  \\
  $2\,13\,14\,44$  & $2\,13\,41\,44$  & $2\,14\,15\,44$  & $2\,15\,32\,44$  & $2\,32\,35\,44$  & $2\,35\,38\,44$  & $2\,38\,41\,44$  \\
  $3\,22\,23\,44$  & $3\,22\,42\,44$  & $3\,23\,24\,44$  & $3\,24\,33\,44$  & $3\,33\,36\,44$  & $3\,36\,39\,44$  & $3\,39\,42\,44$  \\
  $4\,5\,22\,44$   & $4\,12\,14\,44$  & $4\,12\,40\,44$  &  $4\,13\,14\,44$ & $4\,13\,22\,44$  & $5\,6\,30\,44$   & $5\,22\,30\,44$  \\
  $6\,7\,8\,44$    & $6\,7\,31\,44$   & $6\,8\,30\,44$   & $7\,8\,39\,44$   & $7\,31\,34\,44$  & $7\,34\,39\,44$  & $8\,9\,30\,44$   \\
  $8\,9\,42\,44$   & $8\,39\,42\,44$  & $9\,30\,42\,44$  & $12\,14\,15\,44$ & $12\,15\,17\,44$ & $12\,17\,18\,44$ & $12\,18\,40\,44$ \\ 
  $13\,21\,23\,44$ & $13\,21\,41\,44$ & $13\,22\,23\,44$ & $15\,16\,17\,44$ & $15\,16\,32\,44$ & $16\,17\,37\,44$ & $16\,32\,35\,44$ \\
  $16\,35\,37\,44$ & $17\,18\,40\,44$ & $17\,37\,40\,44$ & $21\,23\,24\,44$ & $21\,24\,26\,44$ & $21\,26\,27\,44$ & $21\,27\,41\,44$ \\
  $22\,30\,42\,44$ & $24\,25\,26\,44$ & $24\,25\,33\,44$ & $25\,26\,38\,44$ & $25\,33\,36\,44$ & $25\,36\,38\,44$ & $26\,27\,41\,44$ \\ 
  $26\,38\,41\,44$ & $34\,37\,39\,44$ & $35\,37\,38\,44$ & $36\,38\,39\,44$ & $37\,38\,43\,44$ & $37\,39\,43\,44$ & $38\,39\,43\,44$ \\
 \addlinespace
\bottomrule
\end{tabular*}
\end{center}
\end{table}

\begin{table}[htp]
\small\centering
\defaultaddspace=0.6em
\caption{The sphere $S_{18,125}$.}\label{tbl:triple_trefoil_sphere}
\begin{center}\footnotesize
\begin{tabular*}{\linewidth}{@{\extracolsep{\fill}}llllllll@{}}
\toprule
 \addlinespace
  $1\,2\,4\,9$    & $1\,2\,4\,15$   & $1\,2\,9\,15$   & $1\,3\,8\,10$    & $1\,3\,8\,12$    & $1\,3\,10\,12$  & $1\,4\,5\,14$    & $1\,4\,5\,16$   \\
  $1\,4\,9\,14$   & $1\,4\,15\,16$  & $1\,5\,7\,11$   & $1\,5\,7\,14$    & $1\,5\,11\,17$   & $1\,5\,12\,16$  & $1\,5\,12\,17$   & $1\,7\,11\,12$  \\
  $1\,7\,12\,16$  & $1\,7\,14\,16$  & $1\,8\,10\,13$  & $1\,8\,12\,17$   & $1\,8\,13\,18$   & $1\,8\,17\,18$  & $1\,9\,14\,15$   & $1\,10\,12\,13$ \\ 
  $1\,11\,12\,18$ & $1\,11\,17\,18$ & $1\,12\,13\,18$ & $1\,14\,15\,16$  & $2\,3\,5\,13$    & $2\,3\,5\,14$   & $2\,3\,13\,14$   & $2\,4\,6\,15$   \\
  $2\,4\,6\,17$   & $2\,4\,9\,17$   & $2\,5\,10\,14$  & $2\,5\,10\,18$   & $2\,5\,13\,18$   & $2\,6\,11\,12$  & $2\,6\,11\,16$   & $2\,6\,12\,15$  \\
  $2\,6\,16\,17$  & $2\,7\,8\,10$   & $2\,7\,8\,11$   & $2\,7\,10\,18$   & $2\,7\,11\,12$   & $2\,7\,12\,16$  & $2\,7\,16\,18$   & $2\,8\,10\,13$  \\
  $2\,8\,11\,16$  & $2\,8\,13\,18$  & $2\,8\,16\,18$  & $2\,9\,12\,15$   & $2\,9\,12\,16$   & $2\,9\,16\,17$  & $2\,10\,13\,14$  & $3\,4\,8\,12$   \\
  $3\,4\,8\,15$   & $3\,4\,10\,12$  & $3\,4\,10\,16$  & $3\,4\,15\,16$   & $3\,5\,7\,13$    & $3\,5\,7\,14$   & $3\,6\,9\,14$    & $3\,6\,9\,18$   \\
  $3\,6\,11\,16$  & $3\,6\,11\,18$  & $3\,6\,14\,17$  & $3\,6\,16\,17$   & $3\,7\,9\,13$    & $3\,7\,9\,18$   & $3\,7\,14\,18$   & $3\,8\,10\,15$  \\  
  $3\,9\,13\,14$  & $3\,10\,15\,17$ & $3\,10\,16\,17$ & $3\,11\,15\,16$  & $3\,11\,15\,17$  & $3\,11\,17\,18$ & $3\,14\,17\,18$  & $4\,5\,10\,14$  \\
  $4\,5\,10\,16$  & $4\,6\,8\,15$   & $4\,6\,8\,17$   & $4\,8\,12\,17$   &  $4\,9\,13\,14$  & $4\,9\,13\,17$  & $4\,10\,12\,13$  & $4\,10\,13\,14$ \\
  $4\,12\,13\,17$ & $5\,6\,7\,8$    & $5\,6\,7\,13$   & $5\,6\,8\,9$     &  $5\,6\,9\,18$   & $5\,6\,13\,18$  & $5\,7\,8\,11$    & $5\,8\,9\,11$   \\
  $5\,9\,10\,16$  & $5\,9\,10\,18$  & $5\,9\,11\,15$  & $5\,9\,12\,15$   & $5\,9\,12\,16$   & $5\,11\,15\,17$ & $5\,12\,15\,17$  & $6\,7\,8\,15$   \\
  $6\,7\,13\,15$  & $6\,8\,9\,14$   & $6\,8\,14\,17$  & $6\,11\,12\,18$  & $6\,12\,13\,15$  & $6\,12\,13\,18$ & $7\,8\,10\,15$   & $7\,9\,10\,17$  \\
  $7\,9\,10\,18$  & $7\,9\,13\,17$  & $7\,10\,15\,17$ & $7\,13\,15\,17$  & $7\,14\,16\,18$  & $8\,9\,11\,14$  & $8\,11\,14\,16$  & $8\,14\,16\,18$ \\
  $8\,14\,17\,18$ & $9\,10\,16\,17$ & $9\,11\,14\,15$ & $11\,14\,15\,16$ & $12\,13\,15\,17$ &                 &                  & \\
 \addlinespace
\bottomrule
\end{tabular*}
\end{center}
\end{table}

\begin{thm}
The $3$-sphere $S_{18,125}$ consists of $125$ tetrahedra and $18$ vertices.
It has face vector $f=(18,143,250,125)$ and contains the triple trefoil knot on three edges in its $1$-skeleton.
\end{thm}

Because of the knot,  $S_{18, 125}$ is not LC. So it cannot admit a discrete Morse with fewer than four critical cells.
However, it does admit a discrete Morse function with one critical vertex, one critical edge, one critical triangle and one critical tetrahedron,
as we once more found by a random search.

\begin{thm}
$S_{18, 125}$ admits a discrete Morse function with one critical vertex, one critical edge, one critical triangle and one critical tetrahedron.
\end{thm}

\pagebreak

\section{Non-evasiveness and vertex-decomposability}

In this section, we show that all vertex-decomposable balls are non-evasive, while the converse is false already in dimension three. 
For example, we show that Rudin's ball is non-evasive, but it is neither vertex-decomposable nor shellable. 
The following Lemma is well known.

\begin{lemma} \label{lem:SheddingBoundary}
Let $v$ be a shedding vertex of a vertex-decomposable $d$-ball $B$. Then
$v$ lies on the boundary of the ball. In particular,
\begin{compactenum}[\rm (i)]
 \item $\link (v, B)$ is a vertex-decomposable  $(d-1)$-ball;
\item $\del (v, B)$ is a vertex-decomposable $d$-ball.
\end{compactenum}
\end{lemma}

\begin{proof}[Proof idea:] 
If $v$ is an interior vertex, then the deletion of $v$ is $d$-dimensional but not $(d-1)$-connected and therefore not vertex-decomposable.
\end{proof}

\begin{thm} \label{thm:VDimpliesNE}
Every vertex-decomposable $d$-ball is non-evasive. In particular, all $2$-balls are non-evasive.
\end{thm}

\begin{proof}
A zero-dimensional vertex-decomposable ball is just a point, so it is indeed non-evasive. Let $B$ be a vertex-decomposable $d$-ball, with $d > 0$. By Lemma~\ref{lem:SheddingBoundary} there is a boundary vertex $v$ such that $\del(v,B)$ is a vertex-decomposable $d$-ball and $\link(v,B)$ is a vertex-decomposable $(d-1)$-ball. The deletion of $v$ from $B$ has fewer facets than $B$, and the link of $v$ in $B$ has smaller dimension than $B$. By double induction on the dimension and the number of facets, we may assume that both $\del(v,B)$ and $\link(v,B)$ are non-evasive. By definition, then, $B$ is non-evasive.
\end{proof}

\enlargethispage*{3mm}

Next, we prove that the converse of Theorem~\ref{thm:VDimpliesNE} above is false.

\begin{thm} \label{thm:Rudin}
Rudin's ball $R$, which has $14$ vertices and $41$ facets, is non-evasive. 
\end{thm}

\begin{proof}
Rudin's ball is given by the following $41$ facets \cite{Rudin}:
\setlength{\arraycolsep}{0.3em}
{\small
\[
\begin{array}{l@{\hspace{3.5mm}}l@{\hspace{3.5mm}}l@{\hspace{3.5mm}}l@{\hspace{3.5mm}}l@{\hspace{3.5mm}}l@{\hspace{3.5mm}}l@{\hspace{3.5mm}}l}
  1\,3\,7\,13,    & 1\,3\,9\,13,   & 1\,5\,7\,11,   & 1\,5\,9\,11,    &
  1\,7\,11\,13,   & 1\,9\,11\,13,  & 2\,4\,8\,14,   & 2\,4\,10\,14,   \\

  2\,6\,8\,12,    & 2\,6\,10\,12,  & 2\,8\,12\,14,  & 2\,10\,12\,14,  &
  3\,4\,7\,11,    & 3\,4\,7\,12,   & 3\,6\,10\,11,  & 3\,6\,10\,14,   \\

  3\,7\,12\,13,   & 3\,7\,11\,14,  & 3\,9\,12\,13,  & 3\,10\,11\,14,  &
  4\,5\,8\,12,    & 4\,5\,8\,13,   & 4\,7\,11\,12,  & 4\,8\,11\,12,   \\

  4\,8\,13\,14,   & 4\,10\,13\,14, & 5\,6\,9\,13,   & 5\,6\,9\,14,    &
  5\,7\,11\,14,   & 5\,8\,12\,13,  & 5\,9\,12\,13,  & 5\,9\,11\,14,   \\

  6\,8\,11\,12,   & 6\,9\,13\,14,  & 6\,10\,11\,12, & 6\,10\,13\,14,  &
  7\,11\,12\,13,  & 8\,12\,13\,14, & 9\,11\,13\,14, & 10\,11\,12\,14, \\

  11\,12\,13\,14. & 
\end{array}
\]
}%
To prove non-evasiveness, we claim that the sequence 
$$(a_1, \ldots, a_{14}) = (3, 4, 5, 12, 13, 1, 7, 9, 14, 8, 11, 10, 2, 6)$$ 
has the following two properties:
\begin{compactenum}[(I)]
\item For each $i \le 5$, 
\[
\link_{a_i} \del_{a_1, \ldots, a_{i-1}} R \hbox{ is a non-evasive $2$-complex;} 
\]
\item $\del_{3, 4, 5, 12, 13} R$ is a non-evasive $2$-complex.
\end{compactenum}
To prove that an arbitrary $2$-complex $C$ with $n$ vertices is non-evasive, we need to find an order $a_1, \ldots, a_k$, $a_{k+1}, \ldots, a_n$ of its vertices so that: 
\begin{compactenum}[(i)]
\item For each $i \le k$, 
\[
\link_{a_i} \del_{a_1, \ldots, a_{i-1}} R \hbox{ is a tree;} 
\]
\item $\del_{a_1, \ldots, a_k} R$ is a tree.
\end{compactenum}
All trees and all simplicial $2$-balls are vertex-decomposable and non-evasive, cf.~Theorem~\ref{thm:VDimpliesNE}. In particular, the link of $3$ in $R$ is a non-evasive $2$-ball. Let us delete this vertex $3$, and proceed with the proof of the claim:

\begin{itemize}
\item The link of $4$ in $\del_3 R$ is the $2$-complex $C$ given by the following $8$ facets
{\small
\[
\begin{array}{l@{\hspace{3.5mm}}l@{\hspace{3.5mm}}l@{\hspace{3.5mm}}l@{\hspace{3.5mm}}l@{\hspace{3.5mm}}l@{\hspace{3.5mm}}l@{\hspace{3.5mm}}l@{\hspace{3.5mm}}l}
  2\,8\,14,  & 2\,10\,14,  & 5\,8\,12,  & 7\,11\,12,  & 8\,11\,12, & 
  8\,13\,14, & 10\,13\,14, & 5\,8\,13. & \\
\end{array}
\]
}%
Let us show that $C$ is non-evasive. The link of $7$ in $C$ is a single edge, 
hence non-evasive. The deletion of $7$ from $C$ yields a complex with the same triangles as $C$, 
except $7\,11\,12$. Inside this smaller complex, the link of $8$ is a path, and the deletion of $8$ yields the $2$-complex
{\small
\[
\begin{array}{l@{\hspace{3.5mm}}l@{\hspace{3.5mm}}l@{\hspace{3.5mm}}l@{\hspace{3.5mm}}l@{\hspace{3.5mm}}l@{\hspace{3.5mm}}l@{\hspace{3.5mm}}l@{\hspace{3.5mm}}l}
  2\,10\,14, & 5\,12, & 11\,12, & 13\,14, & 10\,13\,14, & 5\,13. & &\\
\end{array}
\]
}%
This is a $2$-ball with a $3$-edge path attached, hence non-evasive. In particular, $C$ is non-evasive. 

\item The link of $5$ in $\del_{3,4} R$ is the $2$-complex $D$ given by the following $8$ facets
{\small
\[
\begin{array}{l@{\hspace{3.5mm}}l@{\hspace{3.5mm}}l@{\hspace{3.5mm}}l@{\hspace{3.5mm}}l@{\hspace{3.5mm}}l@{\hspace{3.5mm}}l@{\hspace{3.5mm}}l@{\hspace{3.5mm}}l}
1\,7\,11, &  1\,9\,11, & 6\,9\,13,  & 6\,9\,14, & 7\,11\,14, & 
8\,12\,13, & 9\,11\,14 & 9\,12\,13. & \\
\end{array}
\]
}%
We can delete $8$ first (its link is an edge), then $9$ 
(because its link is a $6$-edge path). The resulting $2$-complex,
{\small
\[
\begin{array}{l@{\hspace{3.5mm}}l@{\hspace{3.5mm}}l@{\hspace{3.5mm}}l@{\hspace{3.5mm}}l@{\hspace{3.5mm}}l@{\hspace{3.5mm}}l@{\hspace{3.5mm}}l@{\hspace{3.5mm}}l}
  1\,7\,11, &  6\,13,  & 6\,14, & 7\,11\,14, & 12\,13, & & & & \\
\end{array}
\]
}%
is a $2$-ball with a $3$-edge path attached, hence non-evasive. 
So $D$ is also non-evasive.

\item The link of $12$ in $\del_{3,4,5} R$ is the (non-pure) $2$-complex $E$ 
given by the following $11$ facets
\setlength{\arraycolsep}{0.4em}
{\small
\[  \!\!\!
\begin{array}{l@{\hspace{3.5mm}}l@{\hspace{3.5mm}}l@{\hspace{3.5mm}}l@{\hspace{3.5mm}}l@{\hspace{3.5mm}}l@{\hspace{3.5mm}}l@{\hspace{3.5mm}}l@{\hspace{3.5mm}}l@{\hspace{3.5mm}}l@{\hspace{3.5mm}}l}
  2\,6\,8,   & 2\,6\,10,  & 2\,8\,14,  & 2\,10\,14, & 6\,8\,11,   & 
  6\,10\,11, & 7\,11\,13, & 8\,13\,14, & 9\,13,     & 10\,11\,14, & 11\,13\,14. \\
\end{array}
\]
}%
We can delete $9$ and $7$, as their links are a point and an edge (respectively); 
after that, we delete $13$, whose link is now a path. The resulting $2$-complex $E'$ has $7$ facets:
{\small
\[ 
\begin{array}{l@{\hspace{3.5mm}}l@{\hspace{3.5mm}}l@{\hspace{3.5mm}}l@{\hspace{3.5mm}}l@{\hspace{3.5mm}}l@{\hspace{3.5mm}}l@{\hspace{3.5mm}}l@{\hspace{3.5mm}}l}
  2\,6\,8,   & 2\,6\,10,   & 2\,8\,14, & 2\,10\,14, & 6\,8\,11, & 
  6\,10\,11, & 10\,11\,14. & \\
\end{array}
\]
}%
The link of $14$ inside $E'$ is a $3$-edge path, and the deletion of $14$ from $E'$ 
yields a (non-evasive) $2$-ball. So, $E'$ and $E$ are non-evasive.

\item The link of $13$ in $\del_{3,4,5,12} R$ is the $2$-complex $F$ given by the following 
$6$ facets
{\small
\[
\begin{array}{l@{\hspace{3.5mm}}l@{\hspace{3.5mm}}l@{\hspace{3.5mm}}l@{\hspace{3.5mm}}l@{\hspace{3.5mm}}l@{\hspace{3.5mm}}l}
  1\,7\,11, & 1\,9\,11, & 6\,9\,14, & 6\,10\,14, & 8\,14, & 9\,11\,14. \\
\end{array}
\]
}%
We can delete $8$ first (its link is a point), then $7$ (its link is single edge). 
The resulting $2$-complex is a $2$-ball. In particular, $F$ is non-evasive.

\enlargethispage*{11mm}

\item Finally, let us examine the $2$-complex $G:=\del_{3,4,5,12,13} R$. 
It consists of $13$ facets:
{\small
\[
\begin{array}{l@{\hspace{3.5mm}}l@{\hspace{3.5mm}}l@{\hspace{3.5mm}}l@{\hspace{3.5mm}}l@{\hspace{3.5mm}}l@{\hspace{3.5mm}}l@{\hspace{3.5mm}}l@{\hspace{3.5mm}}l}
  1\,7\,11,  & 1\,9\,11,  & 2\,6\,8,   & 2\,6\,10,   & 2\,10\,14, &   2\,8\,14,  & 6\,8\,11,  & 6\,9\,14,  & 6\,10\,11,  \\
  6\,10\,14, & 7\,11\,14, & 9\,11\,14, & 10\,11\,14. & & & & & 
\end{array}
\]
}%
From $G$ we can delete $1$ (it has a $2$-edge link), then $7$ ($1$-edge link), 
and then $9$ ($2$-edge link). The resulting $2$-complex $H:=\del_{1,7,9} G$ consists of $8$ facets:
{\small
\[
\begin{array}{l@{\hspace{3.5mm}}l@{\hspace{3.5mm}}l@{\hspace{3.5mm}}l@{\hspace{3.5mm}}l@{\hspace{3.5mm}}l@{\hspace{3.5mm}}l@{\hspace{3.5mm}}l@{\hspace{3.5mm}}l}
  2\,6\,8,   & 2\,6\,10,  & 2\,10\,14, & 2\,8\,14, & 6\,8\,11, & 
  6\,10\,11, & 6\,10\,14, & 10\,11,14. & \\
\end{array}
\]
}%
The link of $14$ inside $H$ is a $4$-edge path, and the deletion from $H$ of $14$ 
yields a $2$-ball. So $H$ is non-evasive; therefore $G$ is non-evasive as well.
\end{itemize}
\end{proof}

\begin{corollary}
Some non-evasive balls are (constructible and) not shellable.
\end{corollary}

For a more general statement on non-evasiveness of convex $3$-balls see [XXX].

\begin{prop}\label{prop:SHnonVD}
Let $B_{7,10}$ be the smallest shellable $3$-ball that is not 
vertex-decomposable~\cite{LutzEGnonVD}. This $B_{7,10}$ is non-evasive.
\end{prop}

\begin{proof}
$B_{7,10}$ is given by the following $10$ tetrahedra:
{\small
\setlength{\arraycolsep}{0.4em}
\[\!\!\!
\begin{array}{l@{\hspace{3.5mm}}l@{\hspace{3.5mm}}l@{\hspace{3.5mm}}l@{\hspace{3.5mm}}l@{\hspace{3.5mm}}l@{\hspace{3.5mm}}l@{\hspace{3.5mm}}l@{\hspace{3.5mm}}l@{\hspace{3.5mm}}l}
  0\,1\,2\,6, & 0\,1\,3\,4, & 0\,1\,3\,6, & 0\,2\,3\,5, & 0\,2\,5\,6, & 0\,3\,5\,6, & 
  1\,2\,4\,5, & 1\,2\,4\,6, & 1\,3\,4\,6, & 2\,4\,5\,6. \\
\end{array}
\]
}%
As explained in ~\cite{LutzEGnonVD}, the deletion of $6$ yields the (non-pure!) $3$-complex $A$ 
given by the facets
{\small
\[
\begin{array}{l@{\hspace{3.5mm}}l@{\hspace{3.5mm}}l@{\hspace{3.5mm}}l@{\hspace{3.5mm}}l}
  0\,1\,2, & 0\,1\,3\,4, & 0\,2\,3\,5, & 1\,2\,4\,5. \\
\end{array}
\]
}%
The link of the vertex $5$ in $A$ consists of two triangles with a point in common; 
this is non-evasive. Deleting $5$ from $A$, we obtain the $3$-complex $B$ with the following facets.
{\small
\[\begin{array}{l@{\hspace{3.5mm}}l@{\hspace{3.5mm}}l@{\hspace{3.5mm}}l@{\hspace{3.5mm}}l}
  0\,1\,2, & 0\,1\,3\,4, & 0\,2\,3, & 1\,2\,4. \\
\end{array}
\]
}%
The link of the vertex $4$ inside $B$ is a triangle with an edge attached, 
hence non-evasive. The deletion of the vertex $4$ from $B$ is a $2$-ball. Therefore, 
$B$ is non-evasive, $A$ is non-evasive, and $B_{7,10}$ is non-evasive as well. 
The sequence of deletions certificating its non-evasiveness is the 
`countdown sequence' $6$--$5$--$4$--$3$--$2$--$1$--$0$.
\end{proof}

\begin{corollary}
Some non-evasive balls are shellable but not vertex-decomposable.
\end{corollary}

\begin{prop} \label{prop:CnonSH}
Let $B_{9,18}$ be the smallest non-shellable $3$-ball, described in~\cite{LutzEGnonSH}. 
$B_{9,18}$ is non-evasive and constructible. 
\end{prop}

\begin{proof}
$B_{9,18}$ is given by the following $18$ tetrahedra:

{\small
\[
\begin{array}{l@{\hspace{3.5mm}}l@{\hspace{3.5mm}}l@{\hspace{3.5mm}}l@{\hspace{3.5mm}}l@{\hspace{3.5mm}}l@{\hspace{3.5mm}}l@{\hspace{3.5mm}}l@{\hspace{3.5mm}}l}
  0\,1\,2\,3, & 0\,1\,2\,4, & 0\,1\,4\,5, & 0\,1\,5\,7, & 0\,1\,6\,8, & 
  0\,1\,7\,8, & 0\,2\,3\,4, & 0\,6\,7\,8, & 1\,2\,3\,6, \\

  1\,2\,4\,5, & 1\,2\,5\,8, & 1\,2\,6\,8, & 1\,5\,7\,8, & 2\,3\,4\,7, & 
  2\,3\,6\,7, & 2\,4\,6\,7, & 2\,4\,6\,8, & 4\,6\,7\,8. 
\end{array}
\]
}%

Consider the $2$-sphere $S$ given by the following $12$ triangles:
{\small
\setlength{\arraycolsep}{0.3em}
\[
\begin{array}{l@{\hspace{3.5mm}}l@{\hspace{3.5mm}}l@{\hspace{3.5mm}}l@{\hspace{3.5mm}}l@{\hspace{3.5mm}}l@{\hspace{3.5mm}}l@{\hspace{3.5mm}}l@{\hspace{3.5mm}}l@{\hspace{3.5mm}}l@{\hspace{3.5mm}}l@{\hspace{3.5mm}}l}
  0\,2\,3, & 0\,2\,4, & 0\,3\,6, & 0\,4\,5, & 0\,5\,7, & 0\,6\,8, & 
  0\,7\,8, & 2\,3\,6, & 2\,4\,5, & 2\,5\,8, & 2\,6\,8, & 5\,7\,8. 
\end{array}
\]
}%
It is easy to see that $S$ minus the triangle $0\,3\,6$ is the same $2$-complex 
as the link of $1$ inside $B_{9,18}$. Since a $2$-sphere minus a triangle yields a $2$-ball, 
and all $2$-balls are shellable, it follows that the link of $1$ inside $B_{9,18}$ 
is shellable. Since shellability is preserved by taking cones, the closed star $C_1$ of $1$ 
inside $B_{9,18}$ is also shellable. Let $B_1:=C_1 \cup 0\,6\,7\,8$. Since $C_1 \cap 0\,6\,7\,8$ 
consists of the two triangles $0\,6\,8$ and $0\,7\,8$, $B_1$ is also shellable. 
(A shelling order for $B_1$ is the shelling order for $C_1$, plus $0\,6\,7\,8$ as last facet.) 
Now, let $B_2$ be the shellable $3$-ball with $7$ vertices (labeled by $0,2,3,4,6,7,8$) 
with the following $6$ facets, already given in a possible shelling order:
{\small
\[
\begin{array}{l@{\hspace{3.5mm}}l@{\hspace{3.5mm}}l@{\hspace{3.5mm}}l@{\hspace{3.5mm}}l@{\hspace{3.5mm}}l}
  0\,2\,3\,4, & 2\,3\,4\,7, & 2\,3\,6\,7, & 2\,4\,6\,7, & 
  2\,4\,6\,8, & 4\,6\,7\,8. 
\end{array}
\]
}%
Clearly, $B_{9,18}$ splits as $B_1 \cup B_2$. Moreover, the intersection $B_1 \cap B_2$ 
is a $2$-ball, given by the following $5$ facets:
{\small
\[
\begin{array}{l@{\hspace{3.5mm}}l@{\hspace{3.5mm}}l@{\hspace{3.5mm}}l@{\hspace{3.5mm}}l}
  0\,2\,3,  & 0\,2\,4, & 2\,3\,6, & 2\,6\,8, & 6\,7\,8. 
\end{array}
\]
}%
In particular, $B_{9,18}$ is constructible. We still have 
to prove that $B$ is non-evasive; we will show this by deleting the vertices 
$1$--$0$--$6$--$3$--$7$--$2$--$4$--$5$--$8$, in this order. 
The link of vertex $1$ in $B_{9,18}$ is the (non-evasive, shellable) 
$2$-ball descrived above. 
The deletion of $1$ from $B_{9,18}$ yields the following $3$-complex $A$:
{\small
\[
\begin{array}{l@{\hspace{3.5mm}}l@{\hspace{3.5mm}}l@{\hspace{3.5mm}}l@{\hspace{3.5mm}}l@{\hspace{3.5mm}}l@{\hspace{3.5mm}}l@{\hspace{3.5mm}}l@{\hspace{3.5mm}}l@{\hspace{3.5mm}}l@{\hspace{3.5mm}}l@{\hspace{3.5mm}}l}
  0\,2\,3\,4, & 0\,6\,7\,8, & 2\,3\,4\,7, & 2\,3\,6\,7, & 2\,4\,6\,7, & 
  2\,4\,6\,8, & 4\,6\,7\,8, & 0\,4\,5,    & 0\,5\,7,    & 2\,4\,5,    & 
  2\,5\,8,    & 5\,7\,8. 
\end{array}
\]
}%
Inside $A$, the link of the vertex $0$ consist of two triangles joined by a $2$-edge path. Such a $2$-complex is clearly non-evasive. Deleting the vertex $0$ from $A$\, we obtain the $3$-complex $B$\, described as follows:
{\small
\[
\begin{array}{l@{\hspace{3.5mm}}l@{\hspace{3.5mm}}l@{\hspace{3.5mm}}l@{\hspace{3.5mm}}l@{\hspace{3.5mm}}l@{\hspace{3.5mm}}l@{\hspace{3.5mm}}l}
  2\,3\,4\,7, & 2\,3\,6\,7, & 2\,4\,6\,7, & 2\,4\,6\,8, & 
  4\,6\,7\,8, & 2\,4\,5,    & 2\,5\,8,    & 5\,7\,8.
\end{array}
\]
}%
Next, we delete $6$, whose link inside $B$ is a $2$-ball with $4$ triangles. The result is this $3$-complex~$C$:
{\small
\[
\begin{array}{l@{\hspace{3.5mm}}l@{\hspace{3.5mm}}l@{\hspace{3.5mm}}l@{\hspace{3.5mm}}ll}
  2\,3\,4\,7, & 2\,4\,5, & 2\,4\,8, & 2\,5\,8, & 4\,7\,8, & 5\,7\,8.
\end{array}
\]
}%
From $C$ we can delete first $3$ (whose link is a triangle) and then $7$ 
(whose link is a $3$-edge path). The result is a $2$-ball, so $C$ is non-evasive. 
As a consequence, $B$, $A$ and $B_{9,18}$ are all non-evasive. 
\end{proof}

Our last result highlights the positive effects of barycentric subdivisions.

\begin{prop} \label{prop:subdivisions}
Let $B$ be a simplicial complex. 
\begin{compactenum}[\em (i)]
\item Although $B_{9,18}$ is not shellable, its barycentric subdivision is vertex-decomposable. 
\item Although $S_{13, 56}$ is not constructible, its barycentric subdivision is vertex-decomposable.
\item Although $B_{12,38}$ is evasive and not LC, its barycentric subdivision is LC and non-evasive. 
\end{compactenum}
\end{prop}

\begin{proof} Sequences of deletions that prove vertex-decomposability of $\sd B_{9,18}$ and  $S_{13, 56}$
                        were found with a computer backtrack search.
Since $B_{12,38}$ is collapsible, by a result of Welker $\sd B_{12, 38}$ is non-evasive \cite{Welker}.  Since $B_{12,38}$ is a collapsible $3$-ball, by a result of the first author $\sd B_{12, 38}$ is locally constructible \cite{Benedetti-DMT4MWB}.
\end{proof}

\begin{corollary}
Some non-evasive balls are (LC and) not constructible.
\end{corollary}

\begin{proof}
The barycentric subdivision of $B_{12,38}$ cannot be constructible by Theorem \ref{thm:A}, because it contains a knotted spanning arc of two edges. 
\end{proof}

\section{Open problems}

The following questions remain open:

\begin{compactitem}
\item Are there constructible $d$-spheres that are not shellable? The problem is open already for $d=3$.
\item Are there non-evasive balls with a knotted spanning edge? 
\item Are there examples of non-shellable spheres that become vertex-decomposable after stacking all facets?
         (This would imply that a non-simplicial $4$-ball can be vertex-decomposable but not shellable.)
\item Are there evasive collapsible $4$-balls?
\item Are there non-evasive balls that are not LC? Are there LC ($3$-)balls that are evasive?
\item Are the $3$-spheres $S_{16,92}$ and $S_{18,125}$  vertex-minimal with the property 
          of having the double trefoil and the triple trefoil knot on three edges in their $1$-skeleton, respectively?
          What happens if we replace the square knot by the granny knot?
\end{compactitem}

\subsection*{Acknowledgements.} Thanks to Jonathan Barmak and Alex Engstr\"{o}m  for helpful discussions.

\begin{small}

\end{small}

\bigskip
\bigskip
\medskip

\small

\noindent
Bruno Benedetti \\
Institut f\"ur Informatik \\
Freie Universit\"at Berlin\\
Takustr.\ 9  \\
14195 Berlin \\
{\tt bruno@zedat.fu-berlin.de}

\vspace{8mm}

\noindent
Frank H. Lutz\\
Institut f\"ur Mathematik\\
Technische Universit\"at Berlin\\
Stra\ss e des 17.\ Juni 136\\
10623 Berlin, Germany\\
{\tt lutz@math.tu-berlin.de}

\end{document}